\documentclass[11pt]{amsart}
\usepackage{amscd,amsmath,amssymb,amsfonts}
\usepackage{url,enumerate,color}
\usepackage{dsfont}
\usepackage[cmtip, all]{xy}

\usepackage[top=30mm,right=30mm,bottom=30mm,left=20mm]{geometry}

\theoremstyle{plain}
\newtheorem{thm}{Theorem}
\newtheorem{lem}[thm]{Lemma}
\newtheorem{cor}[thm]{Corollary}
\newtheorem{prop}[thm]{Proposition}
\theoremstyle{definition}

\newtheorem{rmk}[thm]{Remark}

\numberwithin{thm}{section}
\numberwithin{equation}{thm}

\newcommand{\rank}{{\rm rank}}

\newcommand{\soc}{{\rm soc}}
\newcommand{\Char}{{\rm char}}
\newcommand{\Tr}{{\rm Tr}}
\newcommand{\Gal}{{\rm Gal}}

\newcommand{\Trace}{{\rm Trace}}
\newcommand{\Stab}{{\rm Stab}}
\newcommand{\Ker}{{\rm Ker}}
\newcommand{\sym}{\mathsf{S}}
\newcommand{\alt}{\mathsf{A}}
\newcommand{\Aut}{\mathrm{Aut}}

\newcommand{\Irr}{\mathrm{Irr}}

\newcommand{\eps}{\epsilon}

\newcommand{\diag}{\mathrm{diag}}

\newcommand{\inv}{\mathsf{inv}}

\newcommand{\SL}{\mathrm{SL}}
\newcommand{\PSL}{\mathrm{PSL}}
\newcommand{\GL}{\mathrm{GL}}
\newcommand{\PGL}{\mathrm{PGL}}
\newcommand{\AGL}{\mathrm{AGL}}

\newcommand{\Sp}{\mathrm{Sp}}

\newcommand{\sA}{{\mathcal A}}
\newcommand{\sB}{{\mathcal B}}

\newcommand{\sF}{{\mathcal F}}

\newcommand{\sH}{{\mathcal H}}

\newcommand{\sK}{{\mathcal K}}
\newcommand{\sL}{{\mathcal L}}

\newcommand{\sV}{{\mathcal V}}
\newcommand{\sW}{{\mathcal W}}
\newcommand{\sX}{{\mathcal X}}

\newcommand{\A}{{\mathbb A}}

\newcommand{\C}{{\mathbb C}}
{

\newcommand{\F}{{\mathbb F}}
\newcommand{\G}{{\mathbb G}}

\renewcommand{\P}{{\mathbb P}}
\newcommand{\Q}{{\mathbb Q}}

\newcommand{\Z}{{\mathbb Z}}

\newcommand{\ZB}{{\mathbf Z}}
\newcommand{\CB}{{\mathbf C}}
\newcommand{\NB}{{\mathbf N}}
\newcommand{\OB}{{\mathbf O}}

\newcommand{\bj}{\mathbf{j}}
\newcommand{\St}{\mathsf {St}}
\newcommand{\td}{\tilde{\delta}}

\newcommand{\triv}{{\mathds{1}}}
\newcommand{\geom}{{\rm geom}}
\newcommand{\arith}{{\rm arith}}
\newcommand{\chs}{\,{\mathbf {char}}\,}

\renewcommand{\mod}{\bmod \,}
\renewcommand{\Char}{{\mathsf {Char}}}

\begin{document}
\title{Rigid local systems and finite general linear groups}
\author{Nicholas M. Katz and Pham Huu Tiep}
\address{Department of Mathematics, Princeton University, Princeton, NJ 08544}
\email{nmk@math.princeton.edu}
\address{Department of Mathematics, Rutgers University, Piscataway, NJ 08854}
\email{tiep@math.rutgers.edu}
%\address{Princeton University, Mathematics, Fine Hall, N, tiep@math.rutgers.edu}


\thanks{The second author gratefully acknowledges the support of the NSF (grant DMS-1840702), and the Joshua 
Barlaz Chair in Mathematics.}

\maketitle

\begin{abstract}We use hypergeometric sheaves on $\G_m/\F_q$, which are particular sorts of rigid local systems, to construct explicit local systems whose arithmetic
and geometric monodromy groups are the finite general linear groups 
$\GL_n(q)$ for any $n \ge 2$ and and any prime power $q$, so long as $q > 3$ when $n=2$.
This paper continues a program of finding simple (in the sense of simple to remember) families of exponential sums whose monodromy groups are certain finite groups of Lie type, cf. \cite{Gross},
\cite{KT1}, \cite{KT2}, \cite{KT3} for (certain) finite symplectic and unitary groups, or certain sporadic groups, cf. \cite{KRL}, \cite{KRLT1}, \cite{KRLT2}, \cite{KRLT3}. The novelty of this paper is obtaining $\GL_n(q)$ in this hypergeometric way.
%\edit{Such hypergeometric local systems have been constructed for certain finite symplectic and unitary groups \cite{Gross},
%\cite{KT1} , \cite{KT2}, \cite{KT3}, but no such had been known for central quotients of $\GL_n(q)$ with $n \geq 3$.}
A pullback construction then yields local systems on $\A^1/\F_q$ whose geometric monodromy groups are $\SL_n(q)$. These turn out to recover a construction of Abhyankar.
\end{abstract}

\tableofcontents

\section*{Introduction}
For any integer $n \geq 2$ and any prime power $q$, the finite general linear group $\GL_n(q)$ has a (reducible) 
{\it total Weil representation}, which has a very simple description. It is the action by composition of
$\GL_n(q)$ on the space $W$ of $\C$-valued, or for us $\overline{\Q_\ell}$-valued, functions on the $n$-dimensional $\F_q$ vector space $V:=\F_q^n$. This is a representation of dimension $q^n$. We can split off the delta function $\delta_0$ at $0$, and we are left with the space $W^\star$ of functions on the nonzero vectors $V^\star:=V \setminus \{0\}$. On the set $V^\star$, the group $\F_q^\times$ of invertible scalars acts by homothety, and so the action of $W^\star$ on $V^\star$ breaks in its $q-1$ eigenspaces under the $\F_q^\times$ action.
Thus we have
$$W = \C\delta_0 \oplus \bigoplus_{\chi \in \Irr(\F_q^\times)}W_\chi.$$
 Inside the space $W_\triv$ of $\F_q^\times$-invariant functions (i.e. the space of radial functions) we have the one-dimensional space 
$\C \cdot 1_{V^\star}$ of constant functions, so we have a decomposition
%\footnote{When $n=1$, $W_\triv/\C \cdot 1_{V^\star}=0$.},
 $$W \cong \C \delta_0 \oplus \C \cdot 1_{V^\star} \oplus (W_\triv/\C \cdot 1_{V^\star})\ \oplus \bigoplus_{\chi \neq \triv}W_\chi.$$
 It is easy to see that each $W_\chi$ has dimension $(q^n-1)/(q-1)$, the number of points in the projective space $\P^{n-1}(\F_q)$.
 One knows \cite[Prop. 4.2 (b)]{Ger} that $W_\triv/\C \cdot 1_{V^*}$, and each $W_\chi$ with nontrivial $\chi$, is an irreducible representation of $\GL_n(q)$, called an  {\it irreducible Weil representation} of $\GL_n(q)$. This numerology leads us to search for hypergeometric sheaves of these ranks, indexed by these same $\chi$, for which we can prove first, that they each have finite monodromy, cf. Theorem \ref{individualfinite}, and then that the monodromy of their direct sum,
together with two copies of the trivial representation, is indeed $\GL_n(q)$ in its total Weil representation, cf. 
Theorem \ref{main-sl2}. The individual (irreducible)
hypergeometric sheaves have the images of $\GL_n(q)$ in an irreducible Weil representation as their geometric and arithmetic 
monodromy groups.
A pullback construction then yields local systems on $\A^1/\F_q$ whose geometric monodromy groups are $\SL_n(q)$, and also 
allows us to recover a construction of Abhyankar \cite{Abh}.

\section{The set up}
We work in characteristic $p>0$. We choose a prime $\ell \neq p$, so as to be able to work with $\overline{\Q_\ell}$-cohomology. We fix a nontrivial additive character $\psi$ of $\F_p$, a power $q$ of $p$, and an integer $n \ge 2$. We then define
$$A:=(q^n-1)/(q-1),\ \ B:=(q^{n-1}-1)/(q-1).$$

Recall that given an integer $N \ge 1$ prime to $p$, and a multiplicative character $\rho$, we define
$$\Char(N,\chi):=\{ \mbox{characters }\rho\ {\rm with\ }\rho^N=\chi\}$$
and 
$$\Char(N):=\Char(N,\triv),$$
the group of characters of order dividing $N$.

Our interest with be in the (sheaves geometrically isomorphic to the) following hypergeometric sheaves, indexed by
the multiplicative characters $\chi$ of order dividing $q-1$. We fix a nontrivial additive character $\psi$ of $\F_p$. For the trivial character, we consider
$$\sH_\triv:=\sH yp_\psi(\Char(A) \setminus \triv;\Char(B)), {\rm \ of \ rank\ }A-1.$$
For each nontrivial character $\chi$ of order dividing $q-1$, we consider the hypergeometric sheaf
$$\sH_\chi:=\sH yp_\psi(\Char(A,\chi);\Char(B,\chi),\triv), {\rm \ of \ rank\ }A.$$
\begin{lem}\label{geom-det}If $A$ is odd, then the geometric determinants are given by
$$\det(\sH_\chi)=\sL_\chi.$$
If $A$ is even (possible only when $p$ is odd), then the geometric determinants are given by
$$\det(\sH_\chi)=\sL_{\chi\chi_2},$$
for $\chi_2$ the quadratic character.
\end{lem}
\begin{proof}One knows that the geometric determinant of a hypergeometric sheaf of type $(n,m)$ with $n-m \ge 2$ is the product of the ``upstairs" characters, cf. \cite[8.11.6]{Ka-ESDE}. For $\sH_\triv$, one knows that the product of all (or of all but $\triv$) the elements of $\Char(A)$ is $\triv$ if $A$ is odd, and $\chi_2$ otherwise. For the other $\sH_\chi$, the assertion is that the product of all the elements of $\Char(A,\chi)$ is $\chi\times \prod_{\rho \in \Char(A)}\rho$. To see this, pick one character $\Lambda \in \Char(A,\chi)$. Then the elements of $\Char(A,\chi)$ are precisely the products $\Lambda \rho$ with $\rho \in \Char(A)$, which makes clear that the product is as asserted.
\end{proof}

%\edit{With the exception of $\SL(2,2),\SL(2,3)$, the group $\SL_n(q)$ is perfect, so any complex representation of it has trivial determinant. %This means that we cannot be attaining representations of $\SL_n(q)$, but instead we must be getting certain ``correctly chosen" %representations of $\GL_n(q)$. We need to clarify this point.}

\section{The trace function of $\sH_\triv$}
For any $N \ge 2$ prime to $p$, the Kloosterman sheaf $\sK l_ \psi (\Char(N) \setminus \triv)$ is geometrically isomorphic to the lisse sheaf on $\G_m/\F_p$ whose trace function is given as follows: for $K/\F_p$ a finite extension and $t \in K^\times$, it is 
$$t \mapsto -\sum_{x \in K}\psi_K(Nx -x^N/t),$$
cf. \cite[Lemma 1.2, which concerns $\overline{\psi}$]{KRLT2}.
We also know that $\sK l_\psi(\Char(N))$  is geometrically isomorphic to the lisse sheaf on $\G_m/\F_p$ whose trace function is given as follows: for $K/\F_p$ a finite extension and $t \in K^\times$, it is 
$$t \mapsto \sum_{x \in K,\,x^N=t} \psi_k(Nx),$$
cf. \cite[5.6.2]{Ka-GKM}.

\begin{lem}\label{Htriv}The lisse sheaf on $\G_m/\F_p$ whose trace function is given at $u \in K^\times$ for $K/\F_p$ a finite extension, by
$$u \mapsto \sum_{x \in K,\,y \in K^\times}\psi_K\bigl((-1/u)x^A/y^B +x-y)$$
is geometrically isomorphic to $\sH_\triv$.
\end{lem}
\begin{proof}By defintion, $\sH_\triv$ is the multiplicative $!$ convolution of $\sK l_ \psi (\Char(A) \setminus \triv)$ with the multiplicative inverse of the complex conjugate of $\sK l_\psi(\Char(B))$. Thus for  $u \in K^\times$,  $K/\F_p$ a finite extension, we are looking at
$$u \mapsto \sum_{s, t \in K^\times,\,st=u}\ \sum_{x \in K}\psi_K\bigl(Ax -x^A/t)\sum_{y \in K^\times, y^B=1/s}\psi_K(-Bx).$$
Now use $t = u/s =uy^B$ to write this as
$$\sum_{x \in K,\,y \in K^\times}\psi_K\bigl(Ax -x^A/(uy^B)-By),$$
and note that both $A,B$ are $1$ mod $q$, so $1$ in $\F_p$.
\end{proof}

\section{The trace function of $\sH_\chi$ for $\chi \neq \triv$}
\begin{lem}\label{Hchi}For $\chi$ a nontrivial character of order dividing $q-1$, the lisse sheaf on $\G_m/\F_q$ whose trace function 
at $u \in K^\times$ for $K/\F_q$ a finite extension
is 
%given as follows: for $K/\F_q$ a finite extension and $t \in K^\times$, it is 
$$u \mapsto \sum_{x \in K,\,y \in K^\times}\psi_K((-1/u)x^A/y^B +x-y\bigr)\chi(x/y)$$
is geometrically isomorphic to $\sH_\chi$.
\end{lem}
\begin{proof} Again by \cite[5.6.2]{Ka-GKM}, we know that for $N \ge  1$ prime to $p$, $\sK l _\psi(\Char(N,\chi))$ is geometrically isomorphic to the lisse sheaf on  $\G_m/\F_q$ whose trace function is given as follows: for $K/\F_q$ a finite extension, and $t \in K^\times$, 
$$t \mapsto \sum_{x \in K,\,x^N=t}\psi_K(Nx)\chi_K(x).$$
Applying this with $N=A$ and with $N=B$, we see that $\sH yp_\psi(\Char(A,\chi);\Char(B,\chi))$ is geometrically isomorphic to the lisse sheaf on $\G_m/\F_q$ whose trace function is given as follows: for $K/\F_q$ a finite extension, and $v \in K^\times$,
$$v \mapsto \sum_{s,t \in K,\,st=v}\ \sum_{x \in K,\,x^A=t}\psi_K(Ax)\chi_K(x)\sum_{y \in K,\,y^B=1/s}\psi_K(-Bx)\chi_K(1/y)=$$
$$= \sum_{x,y \in K^\times,\, x^A/y^B=v}\psi_K(Ax-By)\chi_K(x/y) =$$
$$= \sum_{x,y \in K^\times,\, x^A/y^B=v}\psi_K(x-y)\chi_K(x/y),$$
the last equality because both $A,B$ are $1$ mod $q$.
To compute a sheaf geometrically isomorphic to $\sH_\chi$, we must further convolve with $\sH yp(\emptyset;\triv)=\sL_{\psi(-1/x)}$.
So our trace function is given as follows: for $K/\F_q$ a finite extension, and $u \in K^\times$,
$$u \mapsto \sum_{v,w \in K,\,vw=u}\psi_K(-1/w)\sum_{x \in K^\times,\, y \in K^\times,\, x^A/y^B=v}\psi_K(x-y)\chi_K(x/y)=$$
$$=\sum_{x,y \in K^\times}\psi_K\bigl(-(x^A/y^B)/u +x-y\bigr)\chi_K(x/y),$$
the last equality by using $vw=u$ to solve for $-1/w=-v/u$.
Because $\chi$ is nontrivial, the sum does not change if we also allow $x=0$ in the summation.
\end{proof}

\section{Putting it all together}
In the previous sections, we found that for each $\chi$ of order dividing $q-1$, trivial or not, $\sH_\chi$ is geometrically isomorphic to the lisse sheaf on $\G_m/\F_q$ whose trace function is given as follows: for $K/\F_q$ a finite extension, and $u \in K^\times$,
$$t \mapsto \sum_{x \in K,\,y \in K^\times}\psi_K\bigl((-1/u)x^A/y^B +x-y\bigr)\chi_K(x/y).$$
We now make the substitution $(x,y) \mapsto (xy,y)$. Then the above sum becomes
$$t \mapsto \sum_{x \in K,\,y \in K^\times}\psi_K\bigl((-1/u)x^Ay^{A-B} +xy-y\bigr)\chi_K(x).$$

To move to weight zero, we do a Tate twist {\bf (1)}. Concretely, we consider the lisse sheaves on $\G_m/\F_q$, denoted $\sF_\chi$,
whose trace functions are given at $t \in K^\times$ for $K/\F_q$ a finite extension, by
$$\sF_\chi: t \mapsto (1/\#K)\sum_{x \in K,\,y \in K^\times}\psi_K\bigl((-1/u)x^Ay^{A-B} +xy-y\bigr)\chi_K(x).$$
\begin{lem}\label{geomiso}For each $\chi \in \Char(q-1)$, the lisse sheaf $\sF_\chi$  on $\G_m/\F_q$ is geometrically isomorphic to $\sH_\chi$.
\end{lem}
\begin{proof}Immediate from Lemmas \ref{Htriv} and \ref{Hchi}.
\end{proof}
\begin{thm}\label{individualfinite}Each of the sheaves $\sF_\chi$ has finite $G_{\arith}$ and (hence) finite $G_{\geom}$.
\end{thm}
\begin{proof}The key observation is that
$$A-B =q^{n-1}$$
is a power of $q$. Therefore the trace sum of $\sF_\chi$ does not change if we raise some the terms inside the $\psi$ to the $A-B$ power, since this does not alter $\Trace_{K/\F_q}$. Thus the trace sum for $\sF_\chi$ at time $u \in K^\times$ is equal to
$$(1/\#K)\sum_{x \in K,\,y \in K^\times}\psi_K\bigl((-1/u)x^Ay^{A-B} +x^{A-B}y^{A-B}-y^{A-B}\bigr)\chi_K(x).$$
Factoring out the $y^{A-B}$ term, we rewrite this as
$$(1/\#K)\sum_{x \in K}\sum_{y \in K^\times}\psi_K\bigl(y^{A-B}((-1/u)x^A +x^{A-B}-1)\bigr)\chi_K(x).$$
Because $y \mapsto y^{A-B}$ is an automorphism of $K$, so a bijection on $K^\times$, this sum is equal to
$$(1/\#K)\sum_{x \in K}\sum_{y \in K^\times}\psi_K\bigl(y((-1/u)x^A +x^{A-B}-1)\bigr)\chi_K(x).$$
In this sum, which ``makes sense" for $y=0$, the $y=0$ term would be 
$$(1/\#K)\sum_{x \in K}\chi_K(x),$$
which is $1$ for $\chi=\triv$, and $0$ otherwise. So our sum is
$$-\delta_{\triv,\chi} + (1/\#K)\sum_{x \in K}\sum_{y \in K}\psi_K\bigl(y((-1/u)x^A +x^{A-B}-1)\bigr)\chi_K(x).$$
The sum over $y$ is
$$0,\ {\rm unless \ } (-1/u)x^A +x^{A-B}-1=0,\ {\rm in\ which\ case\  it \ is \ }\chi_K(x).$$
Thus the trace of $\sF_\chi$ at time $u \in K^\times$ is
$$if \chi=\triv, -1 + {\rm \ number\ of \  solutions \ }x \in K \ {\rm of \ } (-1/u)x^A +x^{A-B}-1=0,$$
$$if \chi \neq \triv, \sum_{x \in K,\, (-1/u)x^A +x^{A-B}-1=0}\chi_K(x).$$
So in all cases, $\sF_\chi$ has algebraic integer traces, and we are done by \cite[8.14.4, (1) $\iff$ (2) $\iff$ (6)]{Ka-ESDE}.
\end{proof}

\begin{cor}\label{induced}Denote by $f(t)$ the polynomial
$$f(t):=t^B(1-t)^{A-B},$$
and denote by $\inv$ the multiplicative inversion $u \mapsto 1/u$ on $\G_m$.
Then on $\G_m/\F_q$ we have arithmetic isomorphisms
$$f_\star \Q_\ell/\Q_\ell \cong \inv^\star \sF_\triv,$$
and, for each nontrivial $\chi$ of order dividing $q-1$,
$$f_\star \sL_{\overline{\chi}}\cong \inv^\star \sF_\chi.$$
\end{cor}
\begin{proof}The trace function of $\inv^\star \sH_\chi$ attaches to $u \in K^\times$, $K/\F_q$ a finite extension, the sum
$$-\delta_{\triv,\chi} + \sum_{x \in K,\, (-u)x^A +x^{A-B}-1=0}\chi_K(x).$$
The polynomial $(-u)x^A +x^{A-B}-1$ has all its roots nonzero. Dividing through by $x^A$, we may write it as a polynomial in $1/x:=t$. it becomes (remembering that $A-B =q^{n-1}$ is a power of $p$)
$$( (-u)x^A +x^{A-B}-1)/xA =-u +x^{-B}-x^{-A}=-u+t^B-t^A = t^B(1-t)^{A-B}-u.$$
Thus the trace becomes
$$-\delta_{\triv,\chi} + \sum_{t \in K,\,t^B(1-t)^{A-B}=u}\overline{\chi}_K(t),$$
which is precisely the trace function of $f_\star \Q_\ell/\Q_\ell$ for $\chi =\triv$, and of $f_\star \sL_{\overline{\chi}}$ when $\chi \neq \triv$.
Because the sheaves $\sF_\chi$ are each (geometrically, and hence) arithmetically irreducible, this equality of trace functions implies arithmetic isomorphisms of sheaves.
\end{proof}
\begin{cor}\label{weilcandidate}The trace function of $\oplus_{\chi \in \Char(q-1)}\inv^\star \sF_\chi$ at $u \in K^\times$, $K/\F_q$ a finite extension, is
$$-1 +{\rm \ number\ of \  solutions \ }T \in K \ {\rm of \ } T^{(q-1)B}(1-T^{q-1})^{A-B}=u.$$
%$$-1 + {\rm \ number\ of \  solutions \ }x \in K \ {\rm of \ } (-1/u)x^{(q-1)A} +x^{(q-1)(A-B)}-1=0.$$
\end{cor}
\begin{proof}The trace at $u \in K^\times$, $K/\F_q$ a finite extension, is $-1$ plus
$$\sum_{t \in K,\,t^B(1-t)^{A-B}=u}\ \sum_{\chi \in \Char(q-1)}\overline{\chi_K}(t).$$
The sum over $\chi$ vanishes unless $t$ is a $q-1$ power in $K^\times$, in which case we may write $t =T^{q-1}$ for a choice of $q-1$ possible $T \in K^\times$. So the trace is $-1$ plus the number of solutions $T \in K$ of
$$ T^{(q-1)B}(1-T^{q-1})^{A-B}=u.$$
%$$(1/\#K)\sum_{x \in K}\sum_{y \in K}\psi_K\bigl(y[(-1/u)x^A +x^{A-B}-1])\sum_{\chi \in \Char(q-1)}\chi_K(x)=$$
%$$(1/\#K)\sum_{x \in K}\sum_{y \in K}\psi_K\bigl(y[(-1/u)x^{(q-1)A} +x^{(q-1)(A-B)}-1]).$$
%Now repeat the previous $\sum_y$ calculation. 
\end{proof}

From Corollary \ref{induced}, we get
\begin{cor}\label{inducedbis}For $f$ the polynomial $f(t):=t^B(1-t)^{A-B}$, we have an arithmetic isomorphism on $\G_m/\F_q$
$$ f_\star \bigl(\oplus_{\chi \in \Char(q-1)}\sL_\chi\bigr) \cong \overline{\Q_\ell} \oplus \bigl(\oplus_{\chi \in \Char(q-1)}\inv^\star\sF_\chi \bigr)$$
%$$\overline{\Q_\ell} \oplus (\oplus_{\chi \in \Char(q-1)}\inv^\star\sF_\chi )\cong f_\star (\oplus_{\chi \in \Char(q-1)}\sL_\chi).$$
%\edit{Are the two statements the same? Yes, I had forgotten to erase one.}
\end{cor}

In what follows, we will let $\sW(n,q)$ denote the local system $\oplus_{\chi \in \Char(q-1)}\sF_\chi$.

\begin{cor}\label{inducedter}For $F$ the polynomial 
$$F(T):= T^{q^{n-1}-1}-T^{q^n-1},$$
we have an arithmetic isomorphism on $\G_m/\F_q$,
$$F_\star\overline{\Q_\ell}/\overline{\Q_\ell} \cong \inv^\star (\sW(n,q)).$$
\end{cor}

The local system $F_\star\overline{\Q_\ell}/\overline{\Q_\ell} $ lives on $\G_m/\F_p$, and thus provides a descent
of $\sW(n,q)$ to $\G_m/\F_p$. 

\begin{lem}\label{222}
Let $q_0 > 1$ be a power of a prime $p$, $K_0:=\F_{q_0^3}$. 
For each $u \in K_0$, let $N(u)$ denote the number of solutions in $K_0$ of the 
equation $T^{q_0^2}-T^{q_0}=uT$. Then the following statements hold.
\begin{enumerate}[\rm(i)]
\item Suppose $p=2$. Then $N(1) = q_0^2$. Furthermore, $N(u)=q_0$ for exactly $q_0^2$ values of $u \in K_0 \smallsetminus \{1\}$,
and $N(u)=1$ for all the remaining $q_0^3-q_0^2-1$ values of $u \in K_0 \smallsetminus \{1\}$. 
\item Suppose $p > 2$. Then $N(u) = q_0$ for exactly $q_0^2+q_0+1$ values in $u \in K_0$, and
$N(u)=1$ for all the remaining $q_0^3-q_0^2-q_0-1$ values of $u \in K_0$.
\end{enumerate}
\end{lem}

\begin{proof}
Note that $N(u) = \#(\sX_u)+1$, where $\sX_u$ is the set of solutions in $K_0^\times$ of the 
equation $T^{q_0^2}-T^{q_0}=uT$, equivalently, of the equation $T^{q_0^2-1}-T^{q_0-1}=u$. In particular, $K_0^\times$ partitions into
the disjoint union of all $\sX_u$ when $u$ varies over $K_0$, whence
\begin{equation}\label{eq01}
  \sum_{u \in K_0}(N(u)-1) = q_0^3-1.
\end{equation}  

Suppose that $T \in \sX_u$. As $T^{q_0^2}=T^{q_0}+uT$ and $T \in K_0^\times$, we have
$$T=T^{q_0^3}=T^{q_0^2}+u^{q_0}T^{q_0} = T^{q_0}+uT+u^{q_0}T^{q_0} = T^{q_0}(1+u)^{q_0}+uT.$$
Now if $u =-1$ then $2T=0$, which is impossible if $p \neq 2$. On the other hand, if $u \neq -1$, 
then $T^{q_0-1}= (1-u)(1+u)^{-q_0}$. This last equation has at most $q_0-1$ solutions in $K_0^\times$. Conversely,
if $T_0 \in \sX_u$, then $\alpha T_0 \in \sX_u$ for all $\alpha \in \F_{q_0}^\times$. 
Thus we have shown that 
\begin{equation}\label{eq02}
  \mbox{If }u \neq -1, \mbox{ then }N(u) = 1 \mbox{ or }q_0,
\end{equation}  
and that $N(-1)=1$ if $p \neq 2$. In particular, \eqref{eq01} implies (ii) if $p > 2$.

Assume now that $u=1$ and $p=2$. Then 
$$T^{q_0^2}-T^{q_0}-uT = \Tr_{K_0/\F_{q_0}}(T),$$ 
and so $N(1)=q_0^2$. Together with \eqref{eq01}, this also implies that
%\begin{equation}\label{eq02}
$$\sum_{1 \neq u \in K_0}(N(u)-1) = q_0^3-q_0^2,$$
%\end{equation} 
and (i) now follows from \eqref{eq02}.
\end{proof}

\begin{thm}\label{correcttracevalues}
Let $K$ be a finite extension of $\F_p$. Then the following statements holds for the trace at time $u \in K^\times$ on 
$F_\star\overline{\Q_\ell}/\overline{\Q_\ell}$.
\begin{enumerate}[\rm(i)]
\item This trace plus $2$ is a always a $p$-power.
\item If $K \supseteq \F_q$, then this trace is of the form
$q^a -2$ for some integer $0 \le a \le n$.
\item Suppose that $q=p^f$ with $f \geq 2$, and suppose $r$ is a prime divisor of $f$. 
For any prime divisor $r$ of $f$, there exist an extension $K_0$ of $\F_p$ and an element
$u_0 \in K_0^\times$ such that the trace at time $u_0$ is $p^{f/r^c}-2$, where $r^c$ is the $r$-part of $f$.
\end{enumerate}
In particular, {\rm (i)} and {\rm (ii)} hold for the trace at time $u \in K^\times$ on $\sW(n,q)$, now viewed as a local system on $\G_m/\F_p$ via Corollary \ref{inducedter}.
\end{thm}
\begin{proof}
(a) It is equivalent to prove this for $\inv^\star$ of the direct sum sheaf in question.
The trace is $-1$ plus the number of solutions $T \in K$ of
$$T^{(q-1)B}(1-T^{q-1})^{A-B}=u.$$
Write out the polynomial $T^{(q-1)B}(1-T^{q-1})^{A-B}$. It is 
$$T^{(q-1)B}-T^{(q-1)B +(q-1)(A-B)}=T^{q^{n-1}-1}-T^{q^n-1}.$$
So the trace is -1 + the number of solutions $T \in K$ of
$$T^{q^{n-1}-1}-T^{q^n-1}=u.$$
 $T=0$ is visibly not a solution, so the trace is
$$-2 + {\rm \ number\ of \  solutions \ }T \in K \ {\rm of \ }T^{q^{n-1}} -T^{q^n}=uT.$$
The solution set of this last equation, 
\begin{equation}\label{eq11}
  T^{q^{n-1}} -T^{q^n}=uT,
\end{equation}
over any field $K \supseteq \F_p$ forms a vector space over $\F_p$ of finite dimension, hence (i) holds. If $K \supseteq \F_q$,
then the solution set of \eqref{eq11} over $K$ forms an $\F_q$ vector space of dimension $\le n$, so the number of its solutions is indeed $q^a$ for some integer $0 \le a \le n$, yielding (ii).

\smallskip
(b) The rest of the proof is to establish (iii). Write $f=f_0r^c$ and $q_0=p^{f_0}$. 
The idea is to show that for a well chosen prime $s \neq r$, we can take
$$K_0 := \F_{p^{sf_0}}=\F_{q_0^s}.$$
%If $n=2$ (so that $f=r^c$), choose any prime $s \neq r$. 
If $n \geq 3$, then, since $\gcd(n,n-1)=1$, $n(n-1)$ is divisible by at least two distinct primes, so we can find a prime $s \neq r$ that divides exactly one of the two integers $n$ and $n-1$. 
If $n=2$ and $r>2$, we choose $s=2$. If $(n,r)=(2,2)$, choose $s=3$.
With $s$ chosen this way, we choose 
$K_0 := \F_{p^{sf_0}}=\F_{q_0^s}$, and solve the equation \eqref{eq11} over $K_0$, for certain $u \in K_0^\times$. 

First we consider the case $n=r=2$, whence $s=3$. Then $\{nr^c,(n-1)r^c\} = \{2^{c+1},2^c\}$ is congruent
to $\{1,2\}$ modulo $3$ (as a set), and for integers $a,b \ge 0$ we have
%\begin{equation}\label{eq14}
$$q_0^{a+3b}-1 \equiv q_0^a-1 (\mod (q_0^3-1)).$$
%\end{equation}  
If $2|c$ then 
$$T^{q^{n-1}}-T^{q^n} = T^{q_0^{(n-1)r^c}}-T^{q_0^{nr^c}} = T^{q_0}-T^{q_0^2},$$
and if $2\nmid c$ then 
$$T^{q^{n-1}}-T^{q^n} = T^{q_0^{(n-1)r^c}}-T^{q_0^{nr^c}} = T^{q_0^2}-T^{q_0}$$
for all $T \in K_0$. Hence we are done by Lemma \ref{222}.

\smallskip
(c) From now on we may assume that $(n,r) \neq (2,2)$. 
The idea now is to view $K_0:=\F_{q_0^s}$ as vector space over $\F_{q_0}$. The $(q_0-1)^{\mathrm {th}}$ power map 
$$[q_0-1]:x \mapsto x^{q_0-1}$$
maps $K_0^\times$ onto $\mu_{(q_0^s-1)/(q_0-1)} := \{ t \in K_0 \mid t^{(q_0^s-1)/(q_0-1)} = 1\}$,
with fibres the nonzero elements in the $\F_{q_0}$-lines defined by the $(q_0^s-1)/(q_0-1)$ equations
$$T^{q_0}=vT,$$
one for each $v \in \mu_{(q_0^s-1)/(q_0-1)}$.
Conversely, for any $v \in \overline{\F_p}^\times$, the equation $T^{q_0}=vT$ has $q_0$ solutions in $ \overline{\F_p}$, and for such a $T$ we have
$$T^{q_0^i} = v^{(q_0^i-1)/(q_0-1)}T$$
for any $i \in \Z_{\geq 0}$. In particular, $T \in K_0^\times$ if and only if $v$ belongs to $\mu_{(q_0^s-1)/(q_0-1)}$.

For $v \in \mu_{(q_0^s-1)/(q_0-1)}$ and $T$ satisfying $T^{q_0}=vT$, using this last identity and remembering that $q$ is $q_0^{r^c}$, we find that $T^{q^{n-1}}-T^{q^n}=H(v)T$, where
$$H:\mu_{(q_0^s-1)/(q_0-1)} \rightarrow K_0,\ \  v \mapsto v^{(q_0^{(n-1)r^c}-1)/(q_0-1)}- v^{(q_0^{nr^c}-1)/(q_0-1)}.$$

We claim that 
\begin{equation}\label{eq12}
  H\mbox{ is injective when }(n,r) \neq (2,2).
\end{equation}  
Admit this for a moment. Then for each $v \in \mu_{(q_0^s-1)/(q_0-1)}$, the points in the line $T^{q_0}=vT$ are among the $K_0$-solutions of the equation
$${\rm Eqn}(v): T^{q^{n-1}}-T^{q^n}=H(v)T.$$
As the $H(v)$ are pairwise distinct, the nonzero $K_0$-solutions of these $(q_0^s-1)/(q_0-1)$ equations partition $K_0^\times$ into $(q_0^s-1)/(q_0-1)$ disjoint subsets, each of which consists of the $q_0-1$ nonzero points in the line $T^{q_0}=vT$. Therefore each ${\rm Eqn}(v)$ has precisely $q_0$ solutions in $K_0$. Furthermore, since $H(1)=0$, we see that for $v\neq 1$, $v \in \mu_{(q_0^s-1)/(q_0-1)}$, $H(v)\neq 0$. At any such point $u=H(v)$, we then have that the trace
at time $u \in K_0^\times$ is $q_0$, as asserted.

We now prove \eqref{eq12}. 
Suppose then that $H(v)=H(w)$, with $v,w \in \mu_{(q_0^s-1)/(q_0-1)} $, i.e. that we have
\begin{equation}\label{eq13}
  v^{(q_0^{(n-1)r^c}-1)/(q_0-1)}- v^{(q_0^{nr^c}-1)/(q_0-1)}=w^{(q_0^{(n-1)r^c}-1)/(q_0-1)}- w^{(q_0^{nr^c}-1)/(q_0-1)}.
\end{equation}  
As $(n,r) \neq (2,2)$,
$s$ divides exactly one of $n$ and $n-1$. For definiteness, say $s|n$ and $s \nmid (n-1)$.
Then $(q_0^s-1)/(q_0-1)$ divides $(q_0^n-1)/(q_0-1)$, which divides $(q_0^{nr^c}-1)/(q_0-1)$, and hence
$$w^{(q_0^{nr^c}-1)/(q_0-1)}=v^{(q_0^{nr^c}-1)/(q_0-1)} = 1.$$
Thus we have
$$w^{(q_0^{(n-1)r^c}-1)/(q_0-1)}=v^{(q_0^{(n-1)r^c}-1)/(q_0-1)}.$$ 
It follows that the order of $w/v$ divides  
$$\gcd\biggl( \frac{q_0^s-1}{q_0-1},\frac{q_0^{(n-1)r^c}-1}{q_0-1} \biggr) =1$$
since\footnote{Recall that for an integer $a \neq 0,\pm 1$, and positive integers $n,m$ with $\gcd(n,m)=1$, one has $\gcd(a^n-1,a^m-1)=a-1$. as one sees by working in the multiplicative group of $\Z/d\Z$ for any $d$ dividing $\gcd(a^n-1,a^m-1)$. } $\gcd(s,(n-1)r^c)=1$. Thus $w=v$, as asserted.
\end{proof}

\section{Galois groups in this context}
Let $k$ be a field, and $f(t)\in k[t]$ a polynomial whose derivative $f'(t)$ is not identically zero. Recall that the critical values of $f$ are its values at the zeroes of $f'$. On the dense open set
$$U:=\A^1 \setminus \{{\rm critical\ values\ of \ }f\},$$
the sheaf $f_\star \overline{\Q_\ell}$ is lisse, of rank the degree of $f$.

Let us recall the well known identification of $G_{\arith}$ with a Galois group.
\begin{lem}\label{galident}The $G_{\arith}$ of $f_\star \overline{\Q_\ell}$  is the Galois group of the equation
$$f(t)=u$$
over the rational function field $k(u)$. In particular we have 
$$G_{\arith} \subseteq \sym_{{\rm deg}(f)},$$
and the quotient $f_\star \overline{\Q_\ell}/\overline{\Q_\ell}$ has the same $G_{\arith}$, now acting through the deleted permutation representation of $\sym_{{\rm deg}(f)}$.
\end{lem} 
\begin{proof}Indeed, $G_{\arith}$ is the ``monodromy group" of the finite \'{e}tale covering of $U$ defined by
$$f:\A^1 \setminus f^{-1}\{{\rm critical\ values\ of \ }f\}$$
in the sense of \cite[1.2.2]{Ka-LGER}, which is the usual \'{e}tale cohomological incarnation of the Galois group.
\end{proof}

\section{Weil-type representations of special linear groups}
Let $q$ be a power of a prime $p$, $W=\F_q^n$, 
and consider the general linear group $\GL(W) \cong \GL_n(q)$ and the special linear group 
$\SL(W) \cong \SL_n(q)$. These groups act naturally on the set of $q^n$ vectors of $W$, and the corresponding 
permutation character is denoted 
\begin{equation}\label{tau1}
  \tau_n=\tau_{n,q}: g \mapsto q^{\dim_{\F_q}\Ker(g-1_W)}. 
\end{equation}
We will also refer to $\tau_n$ as 
the {\it total Weil character} of $\SL_n(q)$. When $n \geq 3$, $\tau_{n,q}$ decomposes over $\SL_n(q)$ as 
$$\tau_n = 2 \cdot 1_{\SL_n(q)} + \sum^{q-2}_{i=0}\tau^i_n,$$
where $\tau^i_n \in \Irr(\SL_n(q))$ has degree $(q^n-1)/(q-1)-\delta_{i,0}$. We will refer to $\tau^i_n=\tau^i_{n,q}$ as {\it irreducible Weil 
characters} of $\SL_n(q)$.

\begin{thm}\label{main-weil-sl}
Let $p$ be any prime and $q$ be any power of $p$. Let $L=\SL_n(q)$ with $n \geq 3$. Suppose $\psi$ is 
a reducible complex character of $L$ such that
\begin{enumerate}[\rm(a)]
\item $\psi(1)=q^n$;
\item $\psi(g) \in \{q^i \mid 0 \leq i \leq n\}$ for all $g \in L$; 
\item $[\psi,1_L]_L=2$; and 
\item every irreducible constituent of $\psi -2 \cdot 1_L$ is among the $q-1$ irreducible Weil characters $\tau^u_n$, $0 \leq u \leq q-2$, of $L$.
\end{enumerate}
Then $\psi$ is the total Weil character $\tau_n$ of $L$, that is, $\psi=2 \cdot 1_L+\sum^{q-2}_{u=0}\tau^u_n$.
\end{thm}

\begin{proof}
(i) By assumption (d), 
$$\psi = 2 \cdot 1_L + \sum^{q-2}_{u=0}a_u\tau^u_n,$$
where $a_u \in \Z_{\geq 0}$. Comparing the degrees, we obtain
$$1-a_0 = \frac{q^n-1}{q-1}\biggl( q-1-\sum^{q-2}_{u=0}a_u \biggr);$$
in particular, $a_0-1$ is divisible by $(q^n-1)/(q-1)$. On the other hand,
$$-1 \leq a_0-1 \leq \frac{\psi(1)-2}{\tau^0_n(1)}-1 = \frac{q^n-2}{(q^n-q)/(q-1)}-1 \leq \frac{q^3-2}{q^2+q}-1 < \frac{q^n-1}{q-1},$$
since $n \geq 3$. It follows that
\begin{equation}\label{eq:sl21}
  a_0=1,~\sum^{q-2}_{u=1}a_u=q-2.
\end{equation}
In particular, if $2 \leq q \leq 3$, then $\psi=2 \cdot 1_L +\sum^{q-2}_{u=0}\tau^u_n = \tau_n$. Hence we may assume $q \geq 4$.

(ii) Now, view $L$ as $\SL(W)$ where $W=\F_{q}^n$, and consider
the subgroup $H \cong \SL_3(q)$ of $L$ that fixes a $3$-dimensional subspace of $W$ and acts trivially on its complement in $W$. 
The values of $\tau^i_n$ are well known, see e.g. \cite[(1.1)]{T}.  An easy application of this character formula shows that 
\begin{equation}\label{eq:sl22}
  \psi_H = \sum^{q-2}_{u=0}b_u\tau^u_3, \mbox{ where }b_u:=a_u+(q^{n-3}-1),
\end{equation}
in particular, 
%\begin{equation}\label{eq:sl23}
$$b_0 = q^{n-3},~\sum^{q-2}_{u=1}b_u = q^{n-3}(q-2).$$
%\end{equation}
%We will view the subscripts $u$ in $b_u$ as elements of both $\Z/(q-1)\Z$ and $\{0,1, \ldots,q\}$. 
Also, let $\sigma$ be a primitive $(q^2-1)^{\mathrm {th}}$ root of unity in $\overline{\F}_q$, $\delta = \sigma^{q+1}$, 
$\tilde\delta$ be a primitive $(q-1)^{\mathrm {th}}$ root of unity in $\C$, and set
$$\Sigma_k := \sum^{q-2}_{u=1}b_u\tilde\delta^{uk},~~\Delta_k:=\sum^{q-2}_{u=0}a_u\tilde\delta^{uk}$$
for any $k \in \Z$. Note that both $\Sigma_k$ and $\Delta_k$ depend only on $k (\mod (q-1))$. 
Then \eqref{eq:sl21} and \eqref{eq:sl22} imply that
\begin{equation}\label{eq:sl23}
  \Delta_0=q-1,~~|\Delta_k| \leq q-1.
\end{equation}

\smallskip
(iii) The character table of $H$ is well known, see e.g. \cite{SF}. Consider any $k \in \Z$ with $(q-1)/2 \nmid k$. Evaluating $\psi$ at a 
semisimple element in $L$ with eigenvalues $\delta^k,\delta^{-k},1$, we have by (b) that 
%\begin{equation}\label{eq:sl24}
$$\Sigma'_k:= 2q^{n-3}+2b_0+\Sigma_0+\Sigma_k+\Sigma_{-k} = q^{n-2}+\Delta_k+\overline{\Delta_k} 
  \mbox{ belongs to }\sV:=\{q^i \mid 0 \leq i \leq n\}.$$
%\end{equation}  
Next, by adding $q-1$ to $k$ if necessary, which does not change $\Sigma_k$, we may assume that $(q+1) \nmid k$. Evaluating $\psi$ at 
a semisimple element in $L$ with eigenvalues $\delta^{-k},\sigma^k,\sigma^{qk}$ (over $\overline{\F}_q$) and using (b) again, we have that 
%\begin{equation}\label{eq:sl25}
$$\sV \ni 2q^{n-3}+\Sigma_k=q^{n-3}+\Delta_k.$$ 
%\end{equation}  
Thus, for a fixed $k (\mod (q-1))$ with $(q-1)/2 \nmid k$, we can find $a,b \in \Z_{\geq 0}$ such that
%\begin{equation}\label{eq:sl24}
$$q^{n-3}+\Delta_k = q^a,~q^{n-2}+\Delta_k+\overline{\Delta_k} =q^b.$$
%\end{equation}
It follows that $\Delta_k = q^a-q^{n-3} = \overline{\Delta_k}$, and
$$q^b = (q-2)q^{n-3}+2q^a > q^{n-3}.$$
Hence $b \geq n-2$, which in turn implies that $a \geq n-3$. Assume in addition that $a \geq n-2$. Then 
$$\Delta_k \geq q^{n-2}-q^{n-3}.$$
In this case, using $n \geq 3$ and \eqref{eq:sl23}, we obtain that $n=3$ and $a=1$, and so $q^b=3q-2$, which is 
impossible since $q > 2$. We have shown that $a=n-3$, i.e. $\Delta_k=0$.  
 
Thus the polynomial
$$f(t):=\sum^{q-2}_{u=0}a_ut^u \in \Z[t]$$
has $\tilde\delta^k$ with $1 \leq k \leq q-2$, $k \neq (q-1)/2$, as roots. Also, $f(1)=\sum^{q-2}_{u=0}a_u=q-1$ by \eqref{eq:sl21}. 
If $2|q$, it follows that $f(t)$ is divisible by $(t^{q-1}-1)/(t-1)$, and so $f(t)=\sum^{q-2}_{u=0}t^u$. If $2 \nmid q$, we have that
$f(t)$ is divisible by $(t^{q-1}-1)/(t^2-1)$, whence $f(t)=(at+b)(t^{q-3}+t^{q-5}+ \ldots +t^2+1)$ with
$a,b \in \Q$. Evaluating at $t=1$ we obtain $a+b=2$. Next, $b=f(0)=a_0=1$, and so $a=1$, whence $f(t)=\sum^{q-2}_{u=0}t^u$ again.
In other words, $a_u=1$ for all $u$, as stated.
\end{proof}

Theorem \ref{main-weil-sl} will be used in tandem with the following result, which allows us to recognize the size of the ground field $s$ 
for the special linear group $\SL_m(s)$ from the values of a sum of its Weil characters:

\begin{prop}\label{weil-sl-value}
Let $n \geq 3$ and let $q$ be a power of a prime $p$. Let $s$ be a power of a prime $r$, possibly different from $p$, and 
$L=\SL_m(s)$ with $m \geq 2$. Suppose $L$ possesses a reducible complex character of $\psi$ such that
\begin{enumerate}[\rm(a)]
\item $\psi(1)=q^n$;
\item $\psi(g)$ is a $q$-power for any transvection $g \in L$; 
%\item $[\psi,1_L]_L=2$; 
\item $(q^n-1)/(q-1)=(s^m-1)/(s-1)$, and  
\item $\psi =2 \cdot 1_L + \sum^{s-2}_{u=0}a_u\tau^u_m$ is a sum of trivial and irreducible Weil characters $\tau^u_m= \tau^u_{m,s}$
of $L$, $a_u \in \Z_{\geq 0}$, and $a_0=1$.
\end{enumerate}
Then $(m,s)=(n,q)$.
\end{prop}

\begin{proof}
First we note that $m \geq 3$. Indeed, if $m=2$, then (c) implies that $s=q \cdot (q^{n-1}-1)/(q-1)$ is a product of two coprime integers 
larger than $1$, which is impossible since $s$ is a prime power.

By hypothesis, 
$$q^n=\psi(1) = 1 + \frac{s^m-1}{s-1} \cdot \sum^{s-2}_{u=0}a_u = 1 + \frac{q^n-1}{q-1} \cdot \sum^{s-2}_{u=0}a_u,$$
whence $\sum^{s-2}_{u=0}a_u = q-1$.
Next we evaluate $\tau^u_m$ at a transvection $g \in L$ using the character formula \cite[(1.1)]{T}:
$$\tau^u_m = \frac{s^{m-1}-1}{s-1} - \delta_{u,0}.$$
By (b), there exists some $a \in \Z_{\geq 0}$ that 
$$q^a=\psi(g) = 1 + \frac{s^{m-1}-1}{s-1} \cdot \sum^{s-2}_{u=0}a_u = 1+(q-1) \frac{s^{m-1}-1}{s-1}.$$
In particular, $a \geq 2$ since $m \geq 3$. It follows that 
$$s^{m-1}= \frac{s^m-1}{s-1} - \frac{s^{m-1}-1}{s-1} =  \frac{q^n-1}{q-1} -  \frac{q^a-1}{q-1} = q^a \cdot  \frac{q^{n-a}-1}{q-1},$$
and so $s^{m-1}$ is divisible by $q^a$, whence $r=p$. In this case, the $p$-part of $(s^m-1)/(s-1)-1$ is $s$, and the $p$-part of
$(q^n-1)/(q-1)-1$ is $q$, and we conclude using (c) that $s=q$ and $m=n$. 
\end{proof}

We will also need to work with permutation representations of $\SL_n(q)$ of degree $q^n-1$.

\begin{lem}\label{cent-sl}
Let $n \geq 2$ and let $L=\SL(W) \cong \SL_n(q)$ embed in $\sym_N$ via its natural permutation action on the set $\Omega$ of $N:=q^n-1$ 
nonzero vectors of $W = \F_q^n$. Then $\CB_{\sym_N}(L)$ coincides with $\ZB(\GL(W)) \cong C_{q-1}$ acting on $\Omega$.
\end{lem}

\begin{proof}
Clearly, $\ZB(\GL(W))$ commutes with $L$ as subgroups of $\sym(\Omega)$. Conversely, let $h \in \CB_{\sym(\Omega)}(L)$ and 
consider a nonzero $v \in W$. Then $\Stab_L(v)$ has exactly $q-1$ fixed points $\lambda v$, $\lambda \in \F_q^\times$, on $\Omega$. 
As $h$ centralizes $L$, it permutes these $q-1$ fixed points, whence $h(v) = \alpha v$ for some $\alpha \in \F_q^\times$. Now, given
any $0 \neq u \in W$, we can find $g \in L$ such that $u=g(v)$. It follows that 
$$h(u) = h(g(v)) = g(h(v)) = g(\alpha v) = \alpha g(v) = \alpha u,$$
i.e. $h = \alpha \cdot 1_W \in \ZB(\GL(W))$.  
\end{proof}

\begin{lem}\label{trans-sl}
Let $n \geq 2$ and let $L = \SL(W) \cong \SL_n(q)$. Then the following statements hold.
\begin{enumerate}[\rm (i)]
\item Let $P$ be a subgroup of $L$ of index $(q^n-1)/(q-1)$. Then $P$ is either the stabilizer in $L$ of a line or the stabilizer of 
a hyperplane of $W$.
\item Let $Q$ be a subgroup of $L$ of index $q^n-1$, and let $\tau$ denote the transpose-inverse automorphism of $L$ if $n \geq 3$. 
Then either $Q$, or $\tau(Q)$ when $n \geq 3$, is the stabilizer in $L$ of some nonzero vector $v \in W$. 
\end{enumerate}

\end{lem}

\begin{proof}
Our proof uses the following result, which is known in the literature as {\it Borel-Tits theorem}, or {\it Tits' lemma}, cf. \cite[(1.6)]{Se}:
{\it If $G$ is a finite group of Lie type of simply connected type in characteristic $p$ and if $M$ is a maximal subgroup of $G$ containing 
a Sylow $p$-subgroup of $G$, then $M$ is a parabolic subgroup of $G$}.

\smallskip
(i) Let $M$ be a maximal subgroup of $L$ containing $P$. By the above statement, $M$ is a maximal parabolic subgroup of $L$,
that is, there is some $1 \leq i \leq n-1$ such that $M$ is the stabilizer in $L$ of some $i$-dimensional subspace of $W$. Note that
$[L:M]$ is greater than $q(q^n-1)/(q-1)$ if $2 \leq i \leq n-2$, and equal to $(q^n-1)/(q-1)$ if $i \in \{1,n-1\}$. Since $[L:P] = (q^n-1)/(q-1)$,
it follows that $i \in \{1,n-1\}$, and $P=M$.

\smallskip
(ii) Again, let $M$ be a maximal subgroup of $L$ containing $Q$. The arguments in (i) show that $M$ is either the stabilizer in $L$ of a line or the stabilizer of a hyperplane of $W$. Applying $\tau$ to both $Q$ and $M$ when $n \geq 3$ if necessary, we may assume that $M = \Stab_L(\langle v \rangle_{\F_q})$
for some $0 \neq v \in W$.

Note that $M=U \rtimes K$, where $U$ is a normal $p$-subgroup (with $p$ the prime dividing $q$) and $K \cong \GL_{n-1}(q)$, and 
$[M:Q]=q-1$. Hence $U \lhd Q$, $Q = U(Q \cap K)$, and 
\begin{equation}\label{eq:cent1}
  [K:Q \cap K] = [M:Q]=q-1. 
\end{equation}  
If $n=2$, then $|U|=q=|Q|$, whence $Q=U=\Stab_L(v)$, and so we are done in this
case. We will now assume that $n \geq 3$. Now, $\Stab_L(v)=U[K,K]$, where $[K,K] \cong \SL_{n-1}(q)$. By \eqref{eq:cent1},
$[[K,K]:Q \cap [K,K]]$ divides $[K:Q \cap K]=q-1$. On the other hand, the index of proper subgroups of $\SL_{n-1}(q)$ is 
larger than $q-1$, see e.g. \cite[Table 5.2.A]{KlL}, unless $(n,q) = (3,9)$. In the exceptional case, any subgroup of 
$[K,K] \cong \SL_2(9) \cong 2\alt_6$ of index dividing $8$ must coincide with $[K,K]$. Thus in all cases $Q \geq [K,K]$, whence
$Q = U[K,K]$ by order comparison.
\end{proof}

\begin{lem}\label{affine}
Let $r$ be a prime, $m \in \Z_{\geq 2}$, and $(m,r) \neq (2,2)$. Then the affine general linear group ${\mathrm {AGL}}_m(r)$ does not possess any element of order $r^m$.
\end{lem}

\begin{proof}
We can embed ${\mathrm {AGL}}_m(r)$ in $\SL_{m+1}(r)$ (as the stabilizer of a nonzero vector in $\F_r^{m+1}$). Writing any $r$-element
$x \in \SL_{m+1}(r)$ in its Jordan canonical form, we see that $|x|$ is at most $r^a$, where $a = \lceil (m+1)/r \rceil$. As $m \geq 2$ and 
$(m,r) \neq (2,2)$, we have $a \leq m-1$, and the statement follows. 
\end{proof}

We will also need the following classification of certain doubly transitive permutation groups:

\begin{thm}\label{2-trans}
Let $q$ be a prime power, $n \in \Z_{\geq 3}$, $N_0:=(q^n-1)/(q-1)$, and let $\Psi:\sym_{N_0} \to \GL(V_0)$ denote the representation of 
$\sym_{N_0}$ on the deleted permutation module $V_0 = \C^{N_0-1}$. Suppose that $G \leq \sym_{N_0}$ is such that 
$\Psi|_G$ is irreducible and contains an element with simple spectrum. Then one of the following statements holds.
\begin{enumerate}[\rm(i)]
\item $\alt_{N_0} \lhd G \leq \sym_{N_0}$, with $G$ acting naturally on $N_0$ points.
\item $\PSL_m(s) \cong S:=\soc(G) \lhd G \leq \Aut(S)$, for some prime power $s$ and $m \geq 2$, and 
with $S$ acting on $(s^m-1)/(s-1)=(q^n-1)/(q-1)$ lines or hyperplanes of $\F_s^m$. 
\item $(q^n-1)/(q-1)=r^m$ for some prime $r$ and $m \geq 1$, and 
$C_r^m \cong \soc(G) \lhd G \leq \mathrm{AGL}_m(r)$, with $\AGL_m(r)$ acting via affine transformations of $\F_r^m$.
\end{enumerate}
\end{thm}

\begin{proof}
The irreducibility of $\Psi|_G$ is equivalent to that $G$ be a doubly transitive permutation subgroup of $\sym_{N_0}$.
We will apply the classification of finite doubly transitive permutation groups \cite[Theorem 5.3]{Cam}, which is a consequence of 
the classification of finite simple groups. Let $S$ denote the socle $\soc(G)$ of $G$. If $S$ is abelian, then (iii) holds. So we will assume
that $S$ is non-abelian, whence it is a simple group. As (i) holds if $S \cong \alt_{N_0}$, we will also assume that $S \not\cong \alt_{N_0}$.
Furthermore, as in the proof of Proposition \ref{weil-sl-value}, the assumption $n \geq 3$ implies that $N_0-1$ cannot be a prime power.
Direct computation shows that $N_0 \neq 11$, $22$, $23$, $36$, $176$, $276$. Finally, assume we are in
the case where $S \cong \Sp_{2m}(2)$ and $N_0 = 2^{m-1}(2^m \pm 1)$ for some $m \geq 4$. In this case, if $g \in G$ has simple
spectrum, then the order of $g\ZB(G)$, as an element of $G/\ZB(G) \leq \Aut(S)$, is at most $2^{m+1}$ by 
\cite[Theorem 2.16]{GMPS}, which is smaller than $\deg(\Psi) = N_0-1$, a contradiction. 
It now follows from \cite[Theorem 5.3]{Cam} that (ii) holds.
\end{proof}

\begin{rmk}
One may wonder if case (ii) of Theorem \ref{2-trans} may occur for some $s$ coprime to $q$. One such occurrence is 
$2^5-1 = (5^3-1)/4$, and a computation on {\sf {Mathematica}} reveals that this is the only occurence when
$(p^a-1)/(p-1) = (r^b-1)/(r-1)$ for some distinct primes $p, r \leq 1223$ and $2 \leq a,b \leq 200$. If one relaxes the primeness
condition, then this equation is known in literature as the {\it Goormaghtigh equation}, and the only other known solution is
$2^{13}-1 = (90^3-1)/89$.

On the other hand, for any given 
$n \geq 2$, $(q^n-1)/(q-1)$ can be a prime power, or even a prime; in fact, the {\it Bateman--Horn conjecture} 
\cite{BH} implies in particular that this can happen infinitely often. 
\end{rmk}

\smallskip
The main result of this section is the following theorem:

\begin{thm}\label{main-sl1}
Let $q=p^f$ be a power of a prime $p$, $n \geq 3$, and $N:=q^n-1$. Let $G \leq \sym_N$ be a subgroup with the following properties:
\begin{enumerate}[\rm(a)]
\item If $\Phi$ denotes the representation of $\sym_N$ on its natural permutation module $\C^N$, then 
$$\Phi|_G = \oplus^{q-2}_{i=0}\Phi_i \oplus 1_G,$$ 
where $\Phi_i \in \Irr(G)$ has degree $(q^n-1)/(q-1)-\delta_{i,0}$;
\item $G_0:= \Phi_0(G)$ embeds in $\sym_{N_0}$, where $N_0:=(q^n-1)/(q-1)$, in such a way that $\Phi_0$ is the restriction to
$G_0$ of the representation of $\sym_{N_0}$ on its deleted permutation module $\C^{N_0-1}$. Furthermore, {$\Phi_0(G)$ contains 
an element of order $N_0$ and a $p$-subgroup of order $q^{n-1}$}.
\item For every $g \in G$, $\Tr(\Phi(g))+1$ is a $q$-power.
%\item The assumption \edit{\rm{$(\SL)$}} holds for $N_0$.
%\item $(n,q) \neq ***$
\end{enumerate}
Then $\SL(W) \cong L:=G^{(\infty)} \lhd G \leq \GL(W)$ for $W=\F_q^n$. Moreover, $\Phi|_G$ is equivalent to the permutation action of 
$G$ on the set $\Omega$ of nonzero vectors of $W$.
\end{thm}

\begin{proof}
(i) Let $\varphi$, respectively $\varphi_i$, denote the character of $\Phi$, respectively $\Phi_i$. 
Also, let $K$ denote the kernel of $\Phi_0$, so that $G/K \cong G_0$. Note by (a) and faithfulness of $\Phi$ that $K=1$ if $q=2$.
Condition (a) implies that $G$ is a transitive subgroup of 
$\sym_N$, hence $K$ acts semi-transitively on $\Omega$, that is, all $K$-orbits have common length say $k$. 
 As $K = \Ker(\Phi_0)$, we have
$$k=[\varphi|_K,1_K]_K \geq 1+\dim(\Phi_0) = N_0.$$  
On the other hand, $\varphi(x) \leq q^{n-1}-1$ by (c). It follows that 
$$\frac{q^n-1}{q-1} \leq [\varphi|_K,1_K]_K \leq \frac{(q^n-1) + (q^{n-1}-1)(|K|-1)}{|K|},$$
and so
\begin{equation}\label{eq:sl11}
  |K| \leq \frac{q^{n-1}(q-1)^2}{(q^n-1)-(q^{n-1}-1)(q-1)} < (q-1)^2.
\end{equation}

(ii) Now we use (b) and Theorem \ref{2-trans} (applied to $\Psi=\Phi_0$) to determine $G_0$. 
%Condition (b), together with the irreducibility of $\Phi_0$, implies that $G_0$ is in fact 
%a doubly transitive subgroup of $\sym_{N_0}$. Hence we can apply the classification of doubly transitive permutation groups, see e.g.
%\cite{Cam} (which is a consequence of the classification of finite simple groups). We will furthermore assume that $\soc(G_0)$ is a 
%non-abelian simple group -- note that this automatically holds if we invoke condition (d). 
Suppose we are in the affine case, that is $N_0 = r^m$ for some prime $r$ and $C_r^m \cong \soc(G_0) \lhd G_0 \leq \AGL_m(r)$. 
If $m \geq 2$ and $(m,r) \neq (2,2)$, then $G_0$ does not contain any element of order $r^m$ by Lemma \ref{affine}, and this 
contradicts (b). The case
$(q^n-1)/(q-1) = 2^2$ is ruled out since $n \geq 3$. 

Finally, assume that $m=1$, so that $r=(q^n-1)/(q-1)$. Since $G_0 \leq \AGL_1(r)$ is a doubly transitive subgroup of $\sym_r$, 
we actually have 
\begin{equation}\label{eq:sl12}
  G_0 = \AGL_1(r) \cong C_r \rtimes C_{r-1}.
\end{equation}  
As $n \geq 3$, we have by \eqref{eq:sl11} that $|K| < r$, whence 
\begin{equation}\label{eq:sl13}
  |G|_r = |G_0|_r = r. 
\end{equation}  

If $q=2$, then we have $r=2^n-1 \geq 7$ and $r-1=2^n-2 \equiv 2 (\mod 4)$. As $K=1$ in this case, we have that $|G|_2 = |G_0|_2 = 2$,
and so the Sylow $2$-subgroups of $G$ have order $2$. On the other hand, $G$ contains a subgroup $P$ of order $2^{n-1} \geq 4$ 
by (b), a contradiction. Hence we may assume $q \geq 3$ in the affine case.  

Let $\bar{g}$ be a generator of 
$\OB_r(G_0) \cong C_r$ and let $g \in G$ be an inverse image of order $r$ of $\bar{g}$. Recall by \eqref{eq:sl12} that 
any power $g^i \neq g$ is conjugate to $g$ in $G/K =G_0$. But $K$ is a normal $r'$-subgroup of $G$, hence any such 
$g^i \neq g$ is also conjugate to $g$ in $G$ by \cite[Lemma 4.11]{TZ3}. Thus $\NB_G(\langle g \rangle)$ acts transitively
on $\langle g \rangle \smallsetminus \{1\}$ and so also on the $r-1$ nontrivial irreducible characters of $\langle g \rangle$. 
It follows that any transitive permutation representation of $\NB_G(\langle g \rangle)$ that is nontrivial on $g$ has degree 
at least $r > |K|$. Applying this remark to the conjugation action of $\NB_G(\langle g \rangle)$ on $K$, we conclude that
$g$ centralizes $K$. 

Recall also that $\varphi_j(1)=N_0=r$ for any $j>0$, hence $\varphi_j$ has $r$-defect $0$ by \eqref{eq:sl13}. It follows that 
$\varphi_j(g)=0$ and $\Phi_j(g)$ is conjugate to $\diag(1,\epsilon,\epsilon^2, \ldots,\epsilon^{r-1})$ (over $\C$) for a primitive
$r^{\mathrm {th}}$ root of unity $\epsilon \in \C$. As $K$ centralizes $g$, $K$ fixes each of $r$ one-dimensional eigenspaces of
$g$ in $\Phi_j$, and so $(\Phi_j)|_K$ is a sum of one-dimensional representations. This holds for every $j>0$, and also for 
$j=0$ as $K=\Ker(\Phi_0$. It follows by faithfulness of $\Phi$ that $K$ is abelian. Now 
$$K\langle g \rangle/K = \langle \bar{g} \rangle = \OB_r(G_0) \lhd G_0,$$
whence $K \langle g \rangle$ is a normal subgroup of index $r-1$ of $G$ by \eqref{eq:sl12}. Also, $K \langle g \rangle$ is abelian,
as $K$ is abelian and $[g,K]=1$. It follows by Ito's theorem \cite[(6.15)]{Is} that the degree of any irreducible character of $G$ divides
$r-1$, and this contradicts the equality $\varphi_1(1)=r$. 

\smallskip
(iii) We have ruled out the affine case, and hence have that $S:=\soc(G_0)$ is a non-abelian simple group. By Theorem \ref{2-trans}, we 
have that $S = \alt_{N_0}$ acting on $N_0$ points, or (up to an automorphism) $S= \PSL_m(s)$ acting on
\begin{equation}\label{eq:sl14}
  N_0=\frac{s^m-1}{s-1}=\frac{q^n-1}{q-1}
\end{equation}  
lines of $\F_s^m$. As in the proof of Proposition \ref{weil-sl-value}, \eqref{eq:sl14} and $n \geq 3$ imply that $m \geq 3$.

Here we handle the case $q=2$; in particular, $K=1$ and $G = G_0$. If $S = \alt_{N_0}$, then we can take $x \in S$ to be a
a $(N_0-2)$-cycles, for which we have $\varphi(x)+1=3$, contradicting (c). Hence $S = \PSL_m(s)$ and \eqref{eq:sl14} holds.
Applying Proposition \ref{weil-sl-value} to $\psi:=\varphi|_S+1_S$,  we see that $(m,s) = (n,q)$, that is, 
$\soc(G)=\SL_n(2)$, and the statement follows.

\smallskip
(iv) From now on we may assume $q \geq 3$; in particular, $N_0 \geq 13$ and $N_0 \neq 15$, whence $m \geq 3$ and 
$(m,s) \neq (3,2)$, $(4,2)$ in \eqref{eq:sl14}. Let $P(S)$ denote the smallest index of proper subgroups of $S$. 
%and let $\dl(S)$ denote the smallest degree of nontrivial projective irreducible complex representations of $S$. 
By \cite[Table 5.2.A]{KlL},
\begin{equation}\label{eq:sl15}
  P(S) = N_0 > (q-1)^2 > |K|.
\end{equation}
Furthermore, using \cite[Lemma 6.1]{GT2}, \cite[Theorem 3.1]{TZ1}, and also \cite{ATLAS} when
$S = \PSL_3(4)$ and $\PSL_4(3)$, we see that
\begin{equation}\label{eq:sl16}
  \begin{array}{c}\mbox{Any nontrivial projective irreducible complex representation of }S\\
  \mbox{ of degree dividing }N_0
  \mbox{ is a linear representation of }\hat{S}\mbox{ of degree }N_0, \end{array} 
\end{equation}
moreover, such a representation exists only when $S = \PSL_m(s)$, in which case
$\hat{S} = \SL_m(s)$.

\smallskip
(v) Recalling $S = \soc(G_0)$, we let $M >K$ be the normal subgroup of $G$ such that $M/K = S$. Certainly, $K$ fixes every 
irreducible character of $K$. Now, \eqref{eq:sl15} implies that the permutation action of $M$ on $\Irr(K)$ is trivial, that is, every
$\alpha \in \Irr(K)$ is $M$-invariant.

Let $i > 0$ and let $\alpha$ be any irreducible constituent of $(\varphi_i)|_M$. By the previous result and by Clifford's theorem, 
$\alpha|_K = c\gamma$ for some $c \in \Z_{\geq 1}$ and $\gamma \in \Irr(K)$. Again by Clifford's theorem, there is some 
projective irreducible complex representation $\Theta$ of $S$ of degree $c$, and note that $c$ divides $\varphi_i(1)=N_0$. Hence,
by \eqref{eq:sl16}, $c=1$ or $c=N_0$.

Suppose we are in the former case: $c=1$. Then $\alpha(1)=\gamma(1) \leq \sqrt{|K|}$, whence $\alpha(1) < q-1$ by 
\eqref{eq:sl11}. On the other hand, 
$$\varphi_1(1)/\alpha(1) \leq [G:M]=|G_0/S|.$$
If $S=\alt_{N_0}$, then we get $(q^n-1)/(q-1) = \varphi_i(1) < 2(q-1)$, a contradiction since $n \geq 3$. Thus $S=\PSL_m(s)$ with
$m \geq 3$. Since $G_0$ is acting doubly transitively on $N_0$ lines and has socle $S$, we have that 
$|G_0/S| \leq \gcd(m,s-1)e$, if $s=r^e$.  It follows that
\begin{equation}\label{eq:sl17}
  N_0=\frac{s^m-1}{s-1} < (q-1) \cdot \gcd(m,s-1)e.
\end{equation}  
The assumption $n \geq 3$ implies that $q-1 < \sqrt{N_0}$. If $m \geq 5$, then 
$$\gcd(m,s-1)e \leq s(s-1)/2 < \sqrt{(s^m-1)/(s-1)} = \sqrt{N_0},$$
contradicting \eqref{eq:sl17}. If $m=4$, then as $s \geq 3$ we have 
$$\gcd(m,s-1)e \leq 4(s/2) = 2s < \sqrt{(s^4-1)/(s-1)} = \sqrt{N_0},$$
again a contradiction. Suppose $m=3$. If $s \neq 4,8$, then $s=r^e \geq 3e$ and so
$$\gcd(m,s-1)e \leq 3(s/3) = s < \sqrt{(s^3-1)/(s-1)} = \sqrt{N_0},$$
again a contradiction. We also reach a contradiction with \eqref{eq:sl17} when $s=2,8$.  In the remaining case
$(m,s) = (3,4)$, whence $N_0=21$, implying that $(n,q) = (3,4)$ and \eqref{eq:sl17} is violated again.

We have shown that $c=N_0$, and so $\gamma(1)=\alpha(1)/c \leq 1$. Thus every irreducible constituent of each 
$(\varphi_i)|_K$ has degree $1$ when $i > 0$. The same holds for $i=0$ as $K = \Ker(\Phi_0)$. Thus every irreducible 
constituent of $\Phi|_K$ has degree $1$, whence $K$ is abelian since $\Phi$ is faithful. Now $M$ acts on $K$ via conjugation,
with $K$ acting trivially. Using \eqref{eq:sl15}, we conclude that $M$ acts trivially on $K$, i.e. $K \leq \ZB(M)$.
Also, we have shown that $\alpha(1) = N_0$, i.e. $(\Phi_i)|_M$ is irreducible.

\smallskip
(vi) Now we consider $L:=G^{(\infty)}$. Certainly, $L \lhd M$ (as $S \lhd G_0 \leq \Aut(S)$), and so $K \cap L \leq \ZB(L)$.
Also, since $S$ is simple, we must have that $S$ is a composition factor of $L$, whence $KL=M$ and 
$L/(K \cap L) = KL/K = M/K \cong S$. Thus $L$ is a cover of $S$.

As  $KL=M$ and $K = \Ker(\Phi_0)$, $(\Phi_0)|_L$ is irreducible of degree $N_0-1$. Next, recall that for $i > 0$, 
$(\Phi_i)|_M$ is irreducible of degree $N_0$, and $K \leq \ZB(M)$ acts via scalars in $\Phi_i$. 
Let $d_i$ denote the common degree of irreducible constituents $\Phi_{ij}$ of 
$(\Phi_i)|_L$. If $d_i=1$, then $\Phi_{ij} =1_L$ as $L$ is perfect. Thus $(\Phi_i)|_L$ is trivial, and so $\Phi_i$ cannot be
irreducible over $M=KL$. So $d_i > 1$ and $\Phi_{ij}$ is a nontrivial 
irreducible projective representation of $S$. By \eqref{eq:sl16}, $d_i=N_0$, $S = \PSL_m(s)$, and $(\Phi_i)|_L$ comes from 
a linear irreducible representation of $\hat{S}=\SL_m(s)$. The same is true for $i=0$. Ignoring the faithfulness of $\Phi$ (only in this 
paragraph of the proof), we may therefore replace $L$ by $\hat{S} = \SL_m(s)$. Applying \cite[Theorem 3.1]{TZ1}, we see that 
each $(\Phi_i)|_L$ is a Weil representation. Now we can apply Proposition \ref{weil-sl-value} to $\psi = \varphi|_L+1_L$
to obtain that $(m,s) = (n,q)$, that is $L = \SL_n(q)$. We then apply Theorem \ref{main-weil-sl} to the same $\psi$ to
conclude that $\psi=\tau_n$, the total Weil character of $L$. As $\tau_n$ is faithful, we also see that $G^{(\infty)} = \SL_n(q)$.

Recall that $K = \Ker(\Phi_0)$ centralizes $L$ and that $\Phi_0$ embeds $G/K$ in $\sym_{N_0}$. As $n \geq 3$, we see that 
no element of $G$ can induce a graph automorphism of $L$ (modulo the inner, diagonal, and field automorphisms). Next,
the diagonal automorphisms of $L$ fix each of $\tau^i_n$, but no nontrivial field automorphism can fix $\tau^1_n$ (the one 
corresponding to a faithful character of $\ZB(\GL_n(q))$ when we extend $\tau_n$ to $\GL_n(q)$). Thus $G$ can induce only 
inner and diagonal automorphisms of $L$, that is, 
\begin{equation}\label{eq:sl18}
  G/\CB_G(L) \leq \PGL_n(q).
\end{equation}  
 
We now return to the assumption that $G \leq \sym_N$ with $N = q^n-1$. Since $\varphi|_L = \tau_n-1_L$, we see that $L=\SL_n(q)$
acts transitively in the natural permutation action of $\sym_N$. Applying Lemma \ref{trans-sl}(ii), we see that (after twisting with
the inverse-transpose automorphism, equivalently, replacing $W$ by the dual module, if necessary), this is the permutation action
of $L$ on the set $\Omega$ of nonzero vectors of $W=\F_q^n$. 
%Applying Lemma \ref{cent-sl} next, we obtain that 
%\begin{equation}\label{eq:sl19}
%  \CB_{\sym_N}(L) = \ZB(\GL(W)).
%\end{equation}
Consider any $h \in G$. By \eqref{eq:sl18}, the conjugation by $h$ induces an inner-diagonal automorphism of $L = \SL(W)$. 
On the other hand, the action of $L$ on $\Omega$ extends to the natural action of $\GL(W)$ on $\Omega$. Hence we can find 
$h' \in \GL(W) < \sym_N$ such that $h$ and $h'$ induce the same automorphism of $L$. Thus $(h')^{-1}h \in \sym_N$ centralizes 
$L$, whence it belongs to $\GL(W)$ by Lemma \ref{cent-sl}. We conclude that $h \in \GL(W)$, i.e. $G \leq \GL(W)$.  
\end{proof}

\section{Weil representations of $\SL_2(q)$}
As before, let $q=p^f$ be a power of a prime $p$, and let $L:=\SL(W) \cong \SL_2(q)$ for $W = \F_q^2$.
To deal with the case $n=2$, we will need some more technical results, which are also interesting in their own right.

\begin{lem}\label{order}
Let $V = \C^{q^2}$ and let $\Phi:G \to \GL(V)$ be a faithful representation such that 
\begin{enumerate}[\rm(a)]
\item $\Tr(\Phi(g)) \in \{1,q,q^2\}$ for all $g \in G$.
\item $\Phi \cong \oplus^{q-2}_{i=0}\Phi_0 \oplus 2 \cdot 1_G$, where the $\Phi_i \in \Irr(G)$ are pairwise inequivalent.
\end{enumerate}
Then $|G| = |\GL_2(q)|$.
\end{lem}

\begin{proof}
Let $a := \# \{g \in G \mid \Tr(\Phi(g))=q \}$ and let $b := \# \{g \in G \mid \Tr(\Phi(g))=1 \}$, so that $|G|=a+b+1$ by (a). 
The assumption (b) implies that 
$$2=[\varphi,1_G]_G = \frac{q^2+aq+b}{a+b+1},~~q+3 = [\varphi,\varphi]_G = \frac{q^4+aq^2+b}{a+b+1},$$
if $\varphi = \Tr(\Phi)$. Solving for $a$ and $b$, we obtain $a=q^3-2q-1$, $b=q^4-2q^3-q^2+3q$, and so 
$|G|= (q^2-1)(q^2-q)=|\GL_2(q)|$.
\end{proof}

The total Weil character $\tau_2=\tau_{2,q}$ of $\GL_2(q)$, cf. \eqref{tau1}, decomposes as 
$2 \cdot 1_{\GL_2(q)}+\sum^{q-2}_{i=0}\tau^i_2$, with $\tau^i_2 \in \Irr(\GL_2(q))$ of degree $q+1-\delta_{u,0}$ and 
pairwise distinct. The smaller-degree character $\tau^0_2$ restricts to the Steinberg character $\St$ of $L$.
Furthermore, if $1 \leq i \leq (q-2)/2$ then $\tau^i_2$ and $\tau^{q-1-i}_2$ restrict to the same 
irreducible character (denoted $\chi_i$ in \cite[\S38]{Do}) of $L=\SL_2(q)$, and those $\lfloor (q-2)/2 \rfloor$ characters 
are pairwise distinct. If $2 \nmid q$, then $(\tau^{(q-1)/2}_2)|_L$ is the sum of two distinct irreducible characters 
(denoted $\xi_1,\xi_2$ in \cite[\S38]{Do}) of degree $(q+1)/2$. We will refer to these characters $\chi_i$, and also $\xi_1,\xi_2$
when $2 \nmid q$, as {\it irreducible Weil characters} of $\SL_2(q)$, and $\tau_2$ (or rather $(\tau_2)|_L$) as the {\it total Weil 
character} of $\SL_2(q)$.

Now we prove an analogue of Theorem \ref{main-weil-sl} for $\SL_2(q)$.

\begin{thm}\label{weil-sl2}
Let $p$ be any prime, $q$ be any power of $p$, $q \geq 4$, and let $L=\SL_2(q)$. Suppose $\psi$ is 
a reducible complex character of $L$ such that
\begin{enumerate}[\rm(a)]
\item $\psi(1)=q^2$;
\item $\psi(g) \in \{q^i \mid 0 \leq i \leq 2\}$ for all $g \in L$; 
\item $[\psi,1_L]_L=2$; and 
\item every irreducible constituent of $\psi -2 \cdot 1_L$ is among the irreducible Weil characters $\St$, $\chi_i$, $0 \leq i \leq (q-2)/2$, 
and also $\xi_1,\xi_2$ when $2 \nmid q$, of $L$.
\end{enumerate}
Then $\psi$ is the total Weil character $\tau_2$ of $L$.
\end{thm}

\begin{proof}
(i) We will use the character tables of $\SL_2(q)$, Theorem 38.1 of \cite{Do} for $2 \nmid q$ and Theorem 38.2 of \cite{Do} for $2|q$.
Write 
\begin{equation}\label{eq31}
  \psi= \left\{ \begin{array}{ll}2 \cdot 1_L + a \cdot \St + \sum^{(q-3)/2}_{i=1}b_i\chi_i + c_1\xi_1+c_2\xi_2, & 2 \nmid q,\\
     2 \cdot 1_L + a \cdot \St + \sum^{(q-2)/2}_{i=1}b_i\chi_i, & 2 |q, \end{array} \right.
\end{equation}  
with all coefficients $a,b_i,c_i \in \Z_{\geq 0}$.  Evaluating $\psi$ at an element $x$ of order $q+1$, we see by (b) that 
$\psi(y)=2-a \leq 2$ is a $q$-power, which is possible only when $a=1$, since $q \geq 3$. As before, let $\td$ denote a 
primitive $(q-1)^{\mathrm {th}}$ root of unity in $\C$.

First suppose that $2|q$. Then $\sum_i b_i = (q^2-q-2)/(q+1)=q-2$ by degree comparison in \eqref{eq31}. Next, we fix 
an element $y \in L$ of order $q-1$, and for $1 \leq l \leq (q-3)/2$ we have 
$$\psi(y^l) = 3 + \sum^{(q-2)/2}_{i=1}b_i\bigl(\td^{il}+\td^{-il}\bigr).$$
It follows that
$$\sum^{(q-2)/2}_{l=1}\psi(y^l) = 3(q-2)/2+\sum^{(q-2)/2}_{i=1}b_i\biggl(\sum^{(q-2)/2}_{l=1}\bigl(\td^{il}+\td^{-il}\bigr)\biggr) = 3(q-2)/2 -
    \sum^{(q-2)/2}_{i=1}b_i=(q-2)/2.$$
As each $\psi(y^l)$ is a $q$-power, we must have that $\psi(y^l)=1$ for all $1 \leq l \leq (q-2)/2$. Thus, the polynomial
$$f(t)=\sum^{(q-2)/2}_{i=1}b_i\bigl( t^{q-1-i}+t^i \bigr) + 2 \in \Q[t]$$
of degree $q-2$ has all $\td^l$, $1 \leq l \leq q-2$ as roots. Since $f(1) = 2\sum^{(q-2)/2}_{i=1}b_i +2=2q-2$, we conclude that
$f(t)=2(t^{q-1}-1)/(t-1)$, i.e. $b_i = 2$ for all $i$, and so $\psi=\tau_2$, as stated.  
  
\smallskip
(ii) Assume now that $2 \nmid q$. Then $\sum_i b_i + (c_1+c_2)/2= (q^2-q-2)/(q+1)=q-2$ by degree comparison in \eqref{eq31}. 
Evaluating $\psi$ at an element $u \in L$ of order $p$ and another element $v \in L$ of order $p$ that is not conjugate to $u$, we obtain
$$\psi(u) = 2+ \sum_i b_i + \frac{c_1+c_2}{2} + \sqrt{\eps q}\frac{c_1-c_2}{2},~\psi(v) = 2+ \sum_i b_i + \frac{c_1+c_2}{2} - 
    \sqrt{\eps q}\frac{c_1-c_2}{2},$$
where $\eps:=(-1)^{(q-1)/2}$. Thus $\psi(u)+\psi(v)=2q$. As each of $\psi(u)$, $\psi(v)$ is a $q$-power, we must have that
$\psi(u)=q$, whence $c_1=c_2=:c$, and so
$$\sum^{(q-3)/2}_{i=1}b_i +c = q-2.$$
Next we evaluate $\psi$ at the central involution $\bj$ of $L$:
$$\psi(\bj) = 2+q+(q+1)\sum_ib_i(-1)^i + c\eps(q+1).$$
In particular, 
$$q^2-\psi(\bj)= \psi(1)-\psi(\bj) = 2(q+1)\bigl(\sum_{2 \nmid i}b_i+\frac{1-\eps}{2}c\bigr)$$
is divisible by $2(q+1)$. On the other hand, $\psi(\bj) \in \{1,q,q^2\}$, so $\psi(\bj) \neq q$, and either 
\begin{equation}\label{eq32a}
  \psi(\bj)=q^2,~\sum_{2 \nmid i}b_i+\frac{1-\eps}{2}c=0,~\sum_{2|i}b_i+\frac{1+\eps}{2}c=q-2,
\end{equation}
or
\begin{equation}\label{eq32b}
  \psi(\bj)=1,~\sum_{2 \nmid i}b_i+\frac{1-\eps}{2}c=\frac{q-1}{2},~\sum_{2|i}b_i+\frac{1+\eps}{2}c=\frac{q-3}{2}.
\end{equation}  
As above, we fix an element $y \in L$ of order $q-1$, and for $1 \leq l \leq (q-3)/2$ we then have 
$$\psi(y^l) = 3 + \sum^{(q-3)/2}_{i=1}b_i\bigl(\td^{il}+\td^{-il}\bigr)+2c(-1)^l.$$
It follows that
$$\begin{aligned}\sum^{(q-3)/2}_{l=1}\psi(y^l) & 
    = 3(q-3)/2+\sum^{(q-3)/2}_{i=1}b_i\biggl(\sum^{(q-3)/2}_{l=1}\bigl(\td^{il}+\td^{-il}\bigr)\biggr) +2c\sum^{(q-3)/2}_{l=1}(-1)^l\\
    &  = 3(q-3)/2 +\sum^{(q-2)/2}_{i=1}b_i\bigl(-1-(-1)^i\bigr) -c(1+\eps)\\
    & = 3(q-3)/2 -2\bigl(\sum_{2|i}b_i+c(1+\eps)/2\bigr).\end{aligned}$$
In the case of \eqref{eq32a}, $1 \leq \sum_l\psi(y^l) = 3(q-3)/2-2(q-2) = (5-q)/2 \leq 0$, a contradiction. Hence \eqref{eq32b} holds,
and we have that $\sum^{(q-3)/2}_{l=1}\psi(y^l) = 3(q-3)/2-(q-3)=(q-3)/2$.     
As each $\psi(y^l)$ is a $q$-power, we must have that $\psi(y^l)=1$ for all $1 \leq l \leq (q-3)/2$. Thus, the polynomial
$$g(t)=\sum^{(q-3)/2}_{i=1}b_i\bigl( t^{q-1-i}+t^i \bigr) + 2ct^{(q-1)/2}+2 \in \Q[t]$$
of degree $q-2$ has all $\td^l \neq \pm 1$, $0 \leq l \leq q-2$, as roots, and so 
$g(t)=(at+b)(t^{q-1}-1)/(t^2-1)$ for some $a,b \in \Q$. Since $b=g(0)=2$ and $(a+b)(q-1)/2=g(1) = 2\sum^{(q-3)/2}_{i=1}b_i +2c+2=2q-2$, we conclude that $a=b=2$, $g(t)=2(t^{q-1}-1)/(t-1)$, i.e. $b_i = 2$ for all $i$ and $c_1=c_2=1$, and so $\psi=\tau_2$, as stated.      
\end{proof}

Recall that a subgroup $Y$ of a group $X$ is a {\it characteristic subgroup} of $X$, $Y \chs X$, if $\phi(Y) \leq Y$ for all
$\phi \in \Aut(X)$.

\begin{prop}\label{sl2-ext}
Let $q \geq 4$ be a prime power and let $X$ be a finite group with a normal subgroup $K$ of order dividing $q-1$ such 
that $X/K \cong S:=\PSL_2(q)$. Then the following statements hold.
\begin{enumerate}[\rm(i)]
\item $K \chs X$.
\item $X$ contains a characteristic subgroup $D$ such that $D$ is quasisimple and $D/\ZB(D) \cong S$.
\end{enumerate}
\end{prop}

\begin{proof}
First we prove (i). Consider any $\phi \in \Aut(X)$. Then $\phi(K) \lhd X$ and so $\phi(K)K/K$ is a normal subgroup of $S$ of order 
dividing $q-1$. As $S$ is simple of order $> q-1$, $\phi(K) = K$, and so $K \chs X$.

To prove (ii), we proceed by induction on $|K|$, with the induction base being trivial. For the induction base, as usual let $P(S)$ denote the smallest index of proper subgroups of $S = \PSL_2(q)$. 
%and let $\dl(S)$ denote the smallest degree of nontrivial projective irreducible complex representations of $S$. 
By \cite[Table 5.2.A]{KlL},
\begin{equation}\label{ext11}
  P(S) \geq q > |K|,
\end{equation}
unless $q=9$, for which we have $P(S)=6$.

\smallskip
(a) First we consider the case where $K$ is abelian. We claim that 
\begin{equation}\label{ext12}
  K \leq \ZB(X). 
\end{equation}
Indeed, $K$ centralizes $K$, and so the conjugation
induces a permutation action of $S=X/K$ on $K \smallsetminus \{1\}$, of size $q-2$. If $q \neq 9$, then \eqref{ext11} implies that 
any transitive permutation of $S$ of degree less than $q$ is trivial, and thus any $S$-orbit on $K$ has length $1$, and so $K \leq \ZB(X)$
as stated. Consider the case $q=9$ and suppose that $K \not\leq \CB_X(K)$. Then 
$K \in \{C_2^3,C_4 \times C_2,C_8\}$, $\CB_X(K) = K$ (as $X/K=S$ is simple), and $S=X/K$ embeds in $\Aut(K)$, which is 
either $\SL_3(2)$ or solvable. This is a contradiction, since $S = \PSL_2(9)$ is simple of order $360$ and $|\SL_3(2)| = 168$.
%The previous argument shows that the $S$-orbits on $K \smallsetminus \{1\}$ have 
%length $1$ or at least $6$. In particular, we are done if $|K|$ divides $4$. If $K$ is abelian of order $8$ but not elementary abelian,
%then $K \cong C_8$ or $C_2 \times C_4$. In the former case, $K$ has $4$ elements of order $8$, $2$ elements of order $4$, and
%$1$ element of order $2$. In the latter case, $K$ has $4$ elements of order $4$ and $3$ elements of order $2$. Thus any $S$-orbit
%on $K \smallsetminus \{1\}$ has length $\leq 4$, whence all of them have length $1$, and we are done again. If $K$ is elementary abelian but 
%not contained in $\ZB(X)$, then $K \cong C_2^3$, $\CB_X(K) = K$ (as $X/K=S$ is simple), and $S=X/K$ embeds in $\Aut(K) = \SL_3(2)$.  
%This is a contradiction, since $S = \PSL_2(9)$ has order $360$ and $|\SL_3(2)| = 168$.

Now we take $D:=X^{(\infty)} \chs X$. Then $D$ has $S$ as a composition factor, and so does $KD$, which contains $K$. Since 
$|K| < |S|$, $S$ must be a composition factor of $KD/K$. It follows that
$KD=X$, $D/(K \cap D) \cong KD/K = X/K =S$. Since $K \cap D \leq \ZB(D)$ by \eqref{ext12}, we see that $D$ is a cover of $S$, as desired.

\smallskip
(b) Now we may assume $K$ is non-abelian, in particular, $C:=\CB_X(K)$ is a proper characteristic subgroup of $X$, and $q \geq 7$. 
We aim to show that $KC=X$. 

Consider the case $q=9$. As $K$ has order dividing $8$ 
and $K$ is non-abelian, $K$ is $D_8$ (dihedral) or $Q_8$ (quaternion). In both cases, $X/C$ embeds in $\Aut(K)$, which is solvable.
As $S$ is a quotient of $X$, it follows that $S$ is a composition factor of $C$, and so the same holds for $KC$. As in (i), we infer from this
that $KC=X$, as stated.

Suppose now that $q \neq 9$. As $C< K$, there exists some $1 \neq x \in K$
such that $Y:=\CB_X(x) < X$. Then $[X:Y] \leq |K|-1 \leq q-2$, and so
$[S:KY/K] = [X:KY] \leq q-2$. Hence \eqref{ext11} implies that $KY=X$. Now $K \cap Y$ is a normal subgroup of $Y$ of 
order dividing $q-1$, and $Y/(K \cap Y) \cong KY/K = X/K = S$. Note that $|K \cap Y| < |K|$, as otherwise $Y=X$, contradicting the choice
of $x$. Hence we may apply the induction hypothesis to $Y$ and find a characteristic subgroup $R$ of $Y$ that is a cover of $S$.  
As $q \neq 4,9$, $R = S$, or $2 \nmid q$ and $R \cong L:=\SL_2(q)$. Now, \eqref{ext11} shows that proper subgroups of $S$ have 
index $\geq q$. The same also holds for $L$. (Assume the contrary: $M < L$ and $[L:M] < q$. Then $q > [L:MZ] = [S:MZ/Z]$ for
$Z := \ZB(L)$, and so $MZ=L$ by \eqref{ext11}. As $M < L$ and $|Z|=2$, we then have $M \cap Z = 1$, and $L = M \times Z$,  
a contradiction.) Thus in either case proper subgroups of $R$ have index $\geq q$. As $|K| < q$, this implies that all $R$-orbits on
$K$ have length $1$, i.e. $R \leq C$. Now $R$, and so $KR$, admits $S$ as a composition factor. Arguing as above, we see that 
$X=KR=KC$.

\smallskip
(c) We have shown that $KC=X$ for a proper characteristic subgroup $C$ that does not contain $K$. It follows that $|K \cap C| < |K|$,
$C/(K \cap C) \cong KC/K = X/K = S$. By the induction hypothesis applied to $C$, $C$ contains a subgroup $D \chs C$ that is 
a cover of $S$. As $C \chs X$, we conclude that $D \chs X$, as desired.
\end{proof}

Now we can prove the main result of this section:

\begin{thm}\label{main-sl21}
Let $q=p^f \geq 4$ be a power of a prime $p$, and $N:=q^2-1$. Let $G \leq \sym_N$ be a subgroup with the following properties:
\begin{enumerate}[\rm(a)]
\item If $\Phi$ denotes the representation of $\sym_N$ on its natural permutation module $\C^N$, then 
$$\Phi|_G = \oplus^{q-2}_{i=0}\Phi_i \oplus 1_G,$$ 
where $\Phi_i \in \Irr(G)$ have degree $q+1-\delta_{i,0}$ and all pairwise inequivalent.
\item $G_0:= \Phi_0(G)$ embeds in $\sym_{q+1}$ as the subgroup $\PGL_2(q)$ acting on $N_0:=q+1$ lines of
$\F_q^2$, in such a way that $\Phi_0$ is the restriction to
$G_0$ of the representation of $\sym_{q+1}$ on its deleted permutation module $\C^q$. 
\item For every $g \in G$, $\Tr(\Phi(g))+1$ is a $q$-power.
\end{enumerate}
Then $G \cong \GL(W) = \GL_2(q)$ for $W=\F_q^2$. Moreover, $\Phi|_G$ is equivalent to the permutation action of 
$G$ on the set $\Omega$ of nonzero vectors of $W$.
\end{thm}

\begin{proof}
(i) Let $\varphi$, respectively $\varphi_i$, denote the character of $\Phi$, respectively $\Phi_i$. 
Also, let $K$ denote the kernel of $\Phi_0$, so that $G/K \cong G_0$. By Lemma \ref{order}, $|G| = |\GL_2(q)|$,
and by (b), $G_0 \cong \PGL_2(q)$. It follows that
\begin{equation}\label{eq:sl211}
  |K| =q-1.
\end{equation}
Let $S = \soc(G_0) \cong \PSL_2(q)$. Using \cite[Theorem 3.1]{TZ1}, and also \cite{ATLAS} when
$S = \PSL_2(9)$, we can check that
\begin{equation}\label{eq:sl212}
  \begin{array}{c}\mbox{Any nontrivial projective irreducible complex representation of }S\\
  \mbox{ of degree dividing }N_0
  \mbox{ is a linear representation of }L:=\SL_2(q)\\
  \mbox{ of degree }N_0, \mbox{ or }N_0/2 \mbox{ when }2 \nmid q, \mbox{ or }N_0/3 \mbox{ when }q=5.\end{array} 
\end{equation}
Let $M >K$ be the normal subgroup of $G$ such that $M/K = S$; note that 
\begin{equation}\label{eq:sl213}
  |G/M|=\gcd(2,q-1)
\end{equation}  
as $G/K=G_0 \cong \PGL_2(q)$. By Proposition \ref{sl2-ext}, $M$ contains a subgroup
$D \chs M$ that is a cover of $S$. As $M \lhd G$, $D$ is normal in $G$. 

\smallskip
(ii) We also note that $KD=M$ (as $M/K \cong S$ and $S = D/\ZB(D)$). 
Now, as $K = \Ker(\Phi_0)$, $(\Phi_0)|_D$ is irreducible of degree $N_0-1$. 

Recall that, for any $i > 0$, $\Phi_i$ is irreducible of degree $N_0$. Let $d_i$ denote the common
degree of irreducible constituents $\Phi_{ij}$ of $(\Phi_i)|_D$. If $d_i=1$, then $\Phi_{ij}=1_D$ as $D$ is perfect. Thus 
$(\Phi_i)|_D$ is trivial. So every irreducible constituent $\Psi_{ij}$ of $(\Phi_i)|_M$ is now irreducible over $K$, and so has degree
at most $\leq \sqrt{q-1}$ by \eqref{eq:sl211}. Together with \eqref{eq:sl213}, this implies that 
$N_0 \leq \sqrt{q-1}\cdot\gcd(2,q-1) < q+1$, a contradiction. Thus $d_i > 1$.

In the case $q=5$, $|K|=4$, hence $K$ is abelian, and part (a) pf the proof of Proposition \ref{sl2-ext} shows that we can take 
$D=M^{(\infty)} = G^{(\infty)}$ (with the second equality following from \eqref{eq:sl213}) and that $K \leq \ZB(M)$. Thus $K$ acts 
via scalars on $\Phi_i$. Now, \eqref{eq:sl213} shows that every $\Psi_{ij}$ has degree $N_0$ or $N_0/2$ and it is irreducible
over $D$, as $M=KD$. Thus $d_i = N_0$ or $N_0/2$ in this case.

Using \eqref{eq:sl212} for $q \neq 5$, we now see that $d_i = N_0/2$ or $N_0$, 
and that every irreducible constituent of $(\Phi_i)|_D$ comes from 
a linear irreducible representation of $L=\SL_2(q)$. The same is true for $i=0$. Ignoring the faithfulness of $\Phi$ (only in this 
paragraph of the proof), we may therefore replace $D$ by $L = \SL_2(q)$. Applying \cite[Theorem 3.1]{TZ1}, we see that 
each $(\Phi_i)|_L$ is a sum of irreducible Weil representations. Now we can apply Theorem \ref{weil-sl2} to $\psi = \varphi|_L+1_L$
to conclude that $\psi=\tau_2$, the total Weil character of $L$. As $\tau_2$ is faithful, we also see that $D = \SL_2(q)$.

\smallskip
(iii) As $q \geq 4$, at least one irreducible constituent of degree $N_0$ of $\tau_2$ ($\chi_i$ in the notation of Theorem \ref{weil-sl2}, and which corresponds to a faithful character of $\ZB(\GL_2(q))$ when we extend $\chi_i$ to $\GL_2(q)$), is fixed by diagonal automorphisms but not by any nontrivial field automorphism of $D$. Thus $G$ can induce only inner and diagonal automorphisms of $D$, that is, 
\begin{equation}\label{eq:sl214}
  G/\CB_G(D) \leq \PGL_2(q).
\end{equation}  
 We now return to the assumption that $G \leq \sym_N$ with $N = q^2-1$. Since $\varphi|_D= \tau_2-1_D$, we see that $D=\SL_2(q)$
acts transitively in the natural permutation action of $\sym_N$. Applying Lemma \ref{trans-sl}(ii), we see that this is the permutation action
of $D$ on the set $\Omega$ of nonzero vectors of $W=\F_q^2$. 
Consider any $h \in G$. By \eqref{eq:sl214}, the conjugation by $h$ induces an inner-diagonal automorphism of $L = \SL(W)$. 
On the other hand, the action of $L$ on $\Omega$ extends to the natural action of $\GL(W)$ on $\Omega$. Hence we can find 
$h' \in \GL(W) < \sym_N$ such that $h$ and $h'$ induce the same automorphism of $L$. Thus $(h')^{-1}h \in \sym_N$ centralizes 
$L$, whence it belongs to $\GL(W)$ by Lemma \ref{cent-sl}. We conclude that $h \in \GL(W)$, i.e. $G \leq \GL(W) \cong \GL_2(q)$.
Since $|G|=|\GL_2(q)|$, we have that $G = \GL(W)$, as stated.
\end{proof}

\section{The structure of monodromy groups}

\begin{thm}\label{main-sl2}
Let $q$ be a power of a prime $p$, $n \geq 2$, $q \geq 4$ when $n =2$, and let $K$ be an extension of $\F_q$. 
Then for the geometric and arithmetic
monodromy groups $G_{\geom}$ and $G_{\arith}$ of the local system $\sW(n,q)$ over $\G_m/K$ we have
$$G_\geom = G_\arith \cong \GL_n(q),$$
with the groups acting on $\sW(n,q) \oplus \overline{\Q_{\ell}} \cong F_\star\overline{\Q_\ell}$ 
as in its natural permutation action on the set $\Omega$ of nonzero
vectors of the natural module $\F_q^n$. 
\end{thm}

\begin{proof}
(i) Let $G$ denote either of $G_\geom$ and $G_\arith$ when $n \geq 3$, and $G=G_\geom$ when $n=2$, 
and let $\Phi$ denote the representation of $G$ on 
$\sW(n,q) \oplus \overline{\Q_\ell}$. By Corollary \ref{inducedter} and Lemma \ref{galident},
%\edit{an explicit Sawin-like reference using Corollary \ref{inducedter}}, 
$G$ embeds in $\sym_N$ for $N:=q^n-1$ in such a way
that $\Phi$ extends to the representation of $\sym_N$ on its natural permutation module $\C^N$
which we also denote by $\Phi$.

Fix a character $\theta \in \Char(q-1)$ of
order $q-1$ and let $\Phi_i$ denote the representation of $G$ on $\sF_{\theta^i}$. By Corollary \ref{induced} and Lemma \ref{galident}, 
%\edit{another Sawin-like explanation}, 
$\Phi_0(G)$  embeds in $\sym_{N_0}$ for $N_0:=(q^n-1)/(q-1)=A$ in such a way
that $\Phi_0$ extends to the representation of $\sym_{N_0}$ on its deleted natural permutation module $\C^{N_0-1}$. By 
Lemma \ref{geomiso}, $\sF_{\triv}$ is geometrically isomorphic to the hypergeometric sheaf $\sH_\triv$, whence $G \geq G_\geom$ is 
irreducible in $\Phi_0$, and $\Phi_0(G)$ contains an element of order $N_0$ with simple spectrum (namely, the image of 
a generator of $I(0)$). Furthermore, $\Phi_0(G)$ contains a $p$-subgroup of order $A-B=q^{n-1}$ (namely, the image of $P(\infty)$).
If in addition $n=2$, then $\sH_\triv$ is the Gross $\PGL_2(q)$ local system considered in
\cite[\S13]{KT1}, and so $G_0 = \PGL_2(q)$ (acting on $q+1$ lines of $\F_q^2$). 
Next, for any $1 \leq i \leq q-2$, by Lemma \ref{geomiso}, $\sF_{\theta^i}$ is geometrically isomorphic to the hypergeometric sheaf 
$\sH_{\theta^i}$, whence $G \geq G_\geom$ is irreducible in $\Phi_i$ (which has degree $N_0$). Together with 
Theorem \ref{correcttracevalues}, this ensures that $(G,\Phi)$ fulfills all the conditions (a)--(c) of Theorem \ref{main-sl1} if $n \geq 3$.
Applying Theorem \ref{main-sl1} when $n \geq 3$, we obtain 
$$\SL_n(q) \lhd G_\geom \lhd G_\arith \leq \GL_n(q),$$
with the groups acting on $\sW(n,q) \oplus \overline{\Q_{\ell}}$ as in its natural permutation action on the set $\Omega$ of nonzero
vectors of the natural module $W=\F_q^n$. 

When $n=2$, we also apply Lemma \ref{geom-det} to see that the representations $\Phi_i$ of $G$ have distinct determinants and so are pairwise inequivalent for $0 \leq i \leq q-2$, and thus we have fulfilled all the conditions (a)--(c) of Theorem 
\ref{main-sl21}. Applying Theorem \ref{main-sl21}, we obtain $G_\geom=\GL_2(q)$, again with the group acting on $\sW(2,q) \oplus \overline{\Q_{\ell}}$ as in its natural permutation action on the set $\Omega$ of nonzero
vectors of the natural module $W=\F_q^2$. Now applying Lemma \ref{order} to $G=G_\arith$ we see that
$|G_\arith|=|\GL_2(q)|=|G_\geom|$, and so $G_\arith=G_\geom$.

\smallskip
(ii) It remains to show that $G_\geom=\GL_n(q)$ when $n \geq 3$. Here, $L:=\SL_n(q)$ is perfect, whence $\Phi_i(L)$ is trivial. 
Note that 
\begin{equation}\label{eq:d1}
  \GL_n(q) = \langle L,g \rangle,
\end{equation}
for a regular semisimple element $g$ of order $q^n-1$,
which acts on $\Omega$ cyclically and such that $z = g^{N_0}=\delta \cdot 1_W$ is a generator of $\ZB(\GL_n(q))$. 
{Indeed, we can identify $\F_q^n$ with $\F_{q^n}$ to embed $\GL_1(q^n)=\F_{q^n}^\times$ in $GL_n(q)$, and then take for $g$ a generator $\F_{q^n}^\times$.}
Hence, if $\zeta=\zeta_N \in \C^\times$ is a primitive $N^{\mathrm {th}}$ root of unity, then $\Phi(g)$ has simple spectrum, consisting of all powers of $\zeta$. Now $\Phi(z)$ admits all 
%\edit{if we do it this way, then $\tilde\delta$ is just $\delta$} 
powers $\tilde\delta^j$ of $\tilde\delta=\zeta^{N_0}$, $0 \leq j \leq q-2$, as eigenvalues, each with multiplicity $N_0$. Hence the corresponding $z$-eigenspaces $V_j$ are invariant under $\Phi(\GL_n(q))$. By the definition of Weil representations
\cite[(1.1)]{T}, 
$V_0$ is the direct sum of the trivial representation and an irreducible representation, whose character over $L$ is $\tau^0_n$, whereas 
each $V_i$ with $i > 0$ is an irreducible representation, whose character over $L$ is $\tau^j_n$.
Thus the actions of $G_\geom$ on $V_0$, 
%\edit{Here $V_0$ is the $\F_q^\times$-invariant functions on the set of nonzero vectors,
%and $V_0/\C$ is the quotient of $V_0$ by the one-dimensional space of constant functions? If so, the notation is
%ambiguous, maybe instead write $V_0/({\rm constant\ functions})$ or something like this}
$V_1, \ldots,V_{q-2}$ are each equivalent to one of $\Phi_0 \oplus \triv$, and $\Phi_i$, $1 \leq i \leq q-2$.

We will choose $\theta$ so that 
$\theta(\det(g))=\tilde\delta$ and view $\theta$ as a linear character of $\GL_n(q)$ (trivial on $\SL_n(q)$). Then $g$ 
has determinant $(-1)^{A+1}\tilde\delta^j$ on $V_j$.  It follows that 
$\GL_n(q)$ has determinant $\lambda_j:=\theta^j(\chi_2)^{A+1}$ on $V_j$. 
Now letting $t:=[\GL_n(q):G_\geom]$, we have 
by \eqref{eq:d1} that $G_\geom=\langle L,g^t \rangle$ and that $t|(q-1)$. In particular, the image of the determinantal character of
$G_\geom$ has index divisible by $t$ in $\mu_{q-1} = \langle \tilde\delta \rangle < \C^\times$. 

On the other hand, choosing $\chi:=\theta$ if $2 \nmid A$, and $\chi:=\theta\chi_2$ of $2|A$, we see by Lemma \ref{geom-det}
that the determinant of $G_\geom$ on $\sH_\chi$ is exactly $\sL_{\theta}$, and so the determinantal image of $G_\geom$ on 
$\sH_\chi$ is the (full) image $\mu_{q-1}$ of $\theta$.
%\edit{This seems OK, can delete having it as an edit:and so the determinantal image of $G_\geom$ on 
%$\sH_\chi$ is the (full) image $\mu_{q-1}$ of $\theta$}. 
Applying Lemmas \ref{Htriv} and \ref{Hchi}, we see that the same 
is true for the determinantal image of $G_\geom$ on $\sF_\chi$. This can happen only when $t=1$. 
\end{proof}

\begin{cor}\label{main-sl2h}
Let $q$ be a power of a prime $p$, $n \geq 2$, $q \geq 4$ when $n =2$, and let $K$ be an extension of $\F_q$. 
Then, for any $\chi \in \Char(q-1)$, the geometric monodromy group $G_{\geom,\chi}$ of the hypergeometric sheaf $\sH_\chi$ over 
$\G_m/K$ is the image of $\GL_n(q)$ in one of its $q-1$ irreducible Weil representations, of degree $(q^n-1)/(q-1)-\delta_{\chi,\triv}$,
which are among the $q-1$ nontrivial irreducible constituents $\Phi_i$, $0 \leq i \leq q-2$, 
of the permutation action of $\GL_n(q)$ on the set of nonzero vectors of 
$\F_q^n$. In particular, $G_{\geom,\triv} \cong \PGL_n(q)$.
\end{cor}

\begin{proof}
By Lemma \ref{geomiso}, $\sH_{\chi}$ is geometrically isomorphic to the summand $\sF_\chi$ of $\sW(n,q)$. Hence the first statement
follows by applying Theorem \ref{main-sl2}, as $G_{\geom,\chi}$ is now some $\Phi_i(G_\geom)$, $0 \leq i \leq q-2$. Among the
nontrivial irreducible constituents of the total Weil representation of $\GL_n(q)$, the deleted permutation action on the 
lines of $\F_q^n$ is the only one that has degree $(q^n-q)/(q-1) = \rank(\sH_\triv)$, and this representation factors through
$\PGL_n(q)$. Hence this representation must be realized by $\sH_\triv$, 
and the second statement follows.
\end{proof}

Recall that $\sW(n,q)$ on $\G_m/\F_q$ is arithmetically isomorphic to
the local system $F_\star\overline{\Q_\ell}/\overline{\Q_\ell}$, by Corollary \ref{inducedter}. Next we determine the arithmetic 
monodromy group of $F_\star\overline{\Q_\ell}$ on $\G_m/K$ for $K$ any subfield of $\F_q$.

\begin{thm}\label{main-sl3}
Let $q=p^f$ be a power of a prime $p$, $n \geq 2$, $q \geq 4$ when $n=2$, and let $K=\F_{q^{1/e}}$ be a subfield of $\F_q$ for $e|f$. 
Then for the arithmetic monodromy group $G_{\arith,K}$ of the local system $F_\star\overline{\Q_\ell}$ on $\G_m/K$ we have
$$G_{\arith,K} \cong \GL_n(q) \rtimes C_e \leq \GL_{ne}(K),$$
where the cyclic subgroup $C_e$ can be identified with $\Gal(\F_q/K)$, and with the groups acting on $F_\star\overline{\Q_\ell}$ as in its natural permutation action on the set $\Omega_K$ of nonzero
vectors of the natural module $K^{ne}$. 
\end{thm}

\begin{proof}
In the case $e=1$ or $f=1$, the statement follows from Theorem \ref{main-sl2}. Next, $G_{\arith,K}$ is a normal subgroup of 
$G_{\arith,\F_p}$ with cyclic quotient of order dividing $f/e$, and  $G_{\arith,\F_q} = \GL_n(q)$ is a normal subgroup of 
$G_{\arith,K}$ with cyclic quotient of order dividing $e$. Hence it suffices to prove the statement for $e=f > 1$, that is when $K = \F_p$.

By Lemma \ref{galident}, $G:=G_{\arith,\F_p}$ embeds in $\sym_N$ for $N:=q^n-1$ in such a way
that the action of $G$ on $F_\star\overline{\Q_\ell}$ extends to the representation $\Phi$ of $\sym_N$ on its natural permutation module 
$\C^N$. Furthermore, $G$ contains the geometric monodromy group $G_{\geom} = G_{\arith,\F_q}=\GL_n(q)$ as a normal subgroup;
in particular, $L:=G^{(\infty)} \cong \SL_n(q)$, and 
$$L \lhd G \leq \NB_{\sym_N}(L).$$
Note that we can view $\F_q^n$ as $\F_p^{nf}$ and thus embed $L$ acting on $\Omega$ in $\GL_{nf}(p)$ acting on the set
$\Omega_{\F_p}$ of $N$ nonzero vectors of $\F_p^{nf}$. This embedding shows that
\begin{equation}\label{eq21}
  \NB_{\sym_N}(L) \geq \GL_n(q) \rtimes \Gal(\F_q/\F_p) \cong \GL_n(q) \rtimes C_f.
\end{equation}
We claim that in fact equality holds in \eqref{eq21}. Indeed, by Lemma \ref{cent-sl}, $\CB_{\sym_N}(L) < \GL_n(q)$.  Hence, if 
equality does not hold in \eqref{eq21}, then $n \geq 3$ and $\NB_{\sym_N}(L)$ contains an element $h$ that induces the transpose-inverse 
automorphism of $L$. On the other hand, the induced permutation action of $L$ on the orbit-sums over all $L$-orbits in 
$\Omega$ is just the action on $1$-spaces in $\F_q^n$, and this action must be be stabilized by $h$, which is impossible when $n \geq 3$
as $h$ sends the action on $1$-spaces to the action on hyperplanes. 

It remains to show that $G = \NB_{\sym_N}(L)$. Assume the contrary: $G$ has index $j > 1$ in $\NB_{\sym_N}(L)$. 
By the above results, we have $j|f$ and that 
$$G=\GL_n(q) \rtimes \Gal(\F_q/\F_{p^j}) = \GL_{nf/j}(p^j) \cap \NB_{\sym_N}(L).$$
Restricting $\Phi$ down to $G$ via $\GL_{nf/j}(p^j)$, we see that
\begin{equation}\label{eq22}
  \Tr(\Phi(x))+1 \mbox{ is a power of }p^j \mbox{ for all }x \in G.
\end{equation}   
Now we can find a prime divisor $r$ of $j$, and apply Theorem \ref{correcttracevalues}(iii) to get an element 
$g \in G$ with $\Tr(\Phi(g))+1 = p^{f_0}$, where $f_0$ is the $r'$-part of $f$. This certainly contradicts \eqref{eq22}.
\end{proof}

\begin{cor}\label{main-sl2s}
Let $q$ be a power of a prime $p$, $n \geq 2$, $q \geq 4$ when $n =2$, and let $K$ be an extension of $\F_q$. 
Then, for any divisor $d$ of $q-1$, the geometric 
%and arithmetic
monodromy group $G_{\geom,d}$ of the $[d]^\star$ Kummer pullback of the local system $\sW(n,q)$ on $\G_m/K$ is
the subgroup $\SL_n(q) \rtimes C_{(q-1)/d}$ of $\GL_n(q)$ {\rm (}with $ C_{(q-1)/d}$ being the cyclic group of diagonal matrices
$\diag(x,1,\ldots,1)$ where $x \in \mu_{(q-1)/d}$, or equivalently, is the subgroup of $\GL_n(q)$ on which $\det^{(q-1)/d}=1${\rm )}.
%$$G'_\geom \cong \SL_n(q).$$
%with the groups acting on $\sW(n,q) \oplus \overline{\Q_{\ell}} \cong F_\star\overline{\Q_\ell}$ 
%as in its natural permutation action on the set $\Omega$ of nonzero
%vectors of the natural module $\F_q^n$. 
\end{cor}

\begin{proof}
In the case $d=1$, the statement  holds by Theorem \ref{main-sl2}: $G_{\geom,1} = G_\geom=\GL_n(q)$.
Next we prove the statement for $d=q-1$. 
When we do any $[N]$ Kummer pullback, with $N$ prime to $p$,
the new $G_{\geom,N}$ after the pullback is a normal subgroup of the original
$G_\geom$, such that $G_\geom/G_{\geom,N}$ is cyclic of order dividing $N$.
[When $N|(q-1)$, $K \supseteq \F_q$ contains $\mu_N$, and so 
the same statement is also true for $G_{\arith}$.] In particular, 
$$G_{\geom,q-1} \geq [G_\geom,G_\geom] = \SL_n(q).$$
%[The problem is that \F_q(t^{1/N}) is not a Galois extension of \F_q(t)
%unless  \F_q contains \mu_N.] 
Furthermore, in the case of $\sW(n,q)$, by Lemma \ref{geom-det}, 
the geometric determinants of the individual summands $\sF_\chi$, which for some summands have full order $q-1$, 
all become trivial after the $[q-1]$ Kummer pullback. It follows that $G_{\geom,q-1}$ has full index $q-1$,
and so $G_{\geom,q-1}=\SL_n(q)$. 

For any divisor $d$ of $q-1$, 
$$[G_{\geom,1}:G_{\geom,d}] \leq d,~[G_{\geom,d}:G_{\geom,q-1}] \leq (q-1)/d.$$
Since $[G_{\geom,1}:G_{\geom,q-1}]=q-1$, equality must hold in both of these, and the statement follows for $d$.  
%Since $G_\arith/[G_\arith,G_\arith] \cong C_{q-1}$, some summands of $\sW(n,q)$ also have 
%arithmetic determinants of order $q-1$. Arguing as above, we conclude that $G'_\arith = \SL_n(q)$.
\end{proof}

%\edit{[And also that for $K$ a subfield of $\F_q$
%can we have the expected $SL$ extended by $\Gal(\F_q/K)$ as the new $G'_{\arith}$?]}

\section{Relation to work of Abhyankar}
%\begin{rmk}\label{sl-abh}
After the $[q^n -1]$ Kummer pullback, each of the $q-1$ summands $\sF_\chi$ 
of $\sW(n,q)$ becomes lisse on $\A^1/\F_q$, and the entire representation
$\sL_\triv \oplus \sW(n,q)$
is the local system on $\A^1/\F_q$ whose trace at time $v$ is the number of solutions of
$$T^{q^n} -v^{(q-1)q^{n-1}}T^{q^{n-1}}=T.$$
[This can be seen by taking our original equation $T^{q^n} -T^{q^{n-1}}=T/u$,
multiplying through by $u$, then
writing $u=v^{q^n-1}$ and writing the equation in terms of the new variable $vT$.]
Since this pullback is itself the pullback by $[(q^n-1)/(q-1)]$ of $[q-1]^\star \sW(n,q)$, whose $G_{\geom}$ is $\SL_n(q)$ by Corollary
\ref{main-sl2s}, we see that
this pullback continues to have $G_{\geom}=\SL_n(q)$ (since this $G_{\geom}$ is a normal subgroup of index dividing $(q^n-1)/(q-1)$
in $\SL_n(q)$, a group generated by its $p$-Sylow subgroups).

The iterated Frobenius pullback $[q^{n-1}]^\star$  (i.e. the power $q^{n-1}$ in the exponent of the variable $v$), does not alter either $G_{\geom}$ or $G_{\arith}$, so we can instead
look at the new local system on $\G_m/\F_q$, call it $\sA(n)$, whose trace at time $v$ is the number of solutions of 
$$T^{q^n} -v^{q-1}T^{q^{n-1}}=T,$$
and whose $G_{\geom}$ remains $\SL_n(q)$.

This new local system $\sA(n)$ is the $[q-1]$ Kummer pullback of the local system, call it $\sB(n)$, on $\A^1/\F_q$
whose trace at time $v$ is the number of solutions of 
$$T^{q^n} -vT^{q^{n-1}}=T.$$
Thus $G_{\geom,\sA(n)} (= \SL_n(q)$) is a normal subgroup of  $G_{\geom,\sB(n)}$ of index dividing $q-1$. But as $\sB(n)$ is lisse on
$\A^1/\F_q$, its $G_{\geom,\sB(n)}$ is generated by its $p$-Sylow subgroups, and hence has no nontrivial quotients of order dividing $q-1$. Thus  $G_{\geom,\sB(n)}=\SL_n(q)$ as well.
This in turn means that over the rational function field $\overline{\F_q}(v)$,
the Galois group of the equation $T^{q^n} -vT^{q^{n-1}}=T$,
or equivalently, of the equation
$$T^{q^n -1} -vT^{q^{n-1} -1}=1$$
is $\SL_n(q)$. Thus we have recovered case (i) of \cite[Theorem 1.2]{Abh}.

%\edit{However, there does not seem to be any connection between the local systems we have constructed for 
%$\Sp_{2n}(q)$ and $\SU_{n}(q)$, and their Galois realizations obtained by Abhyankar in his papers subsequent to \cite{Abh}: 
%Abhyankar used polynomials of degree of magnitude roughly $q^{2n}$ in both cases, whereas our local systems,
%and their trace functions, have rank of magnitude roughly $q^n$.}
%\end{rmk}

\end{document}